\documentclass[12pt]{amsart}
\usepackage{amssymb}
\usepackage[margin=1in]{geometry}
\usepackage{amsthm}
\usepackage{amstext}
\usepackage{amsbsy}
\usepackage{amscd}
\usepackage{enumerate}
\usepackage{chngpage}
\usepackage{mathtools}
\usepackage{amsmath}
\usepackage{hyperref}
\usepackage{tikz}
\usepackage{stmaryrd}
\usepackage{color}

\newtheorem{thm}{Theorem}[section]
\newtheorem{lemma}[thm]{Lemma} \newtheorem{cor}[thm]{Corollary}
\newtheorem{prop}[thm]{Proposition}
\theoremstyle{definition}
\newtheorem{defn}[thm]{Definition}
\newtheorem{question}[thm]{Question}

\newtheorem{example}[thm]{Example}

\newcommand{\Z}{\mathbb Z}

\newcommand{\Q}{\mathbb Q}
\newcommand{\F}{\mathbb F}
\newcommand{\Fp}{\mathbb{F}_p}

\newcommand{\mb}[1]{\mathbb{#1}}

\title{Is the number of subrings of index $p^e$ in $\Z^n$ polynomial in $p$?}
\author{Kelly Isham}
\address{214 McGregory Hall, Colgate University, Hamilton, NY 13346}
\begin{document}
	\maketitle

\begin{abstract}
	It is well-known that for each fixed $n$ and $e$, the number of subgroups of index $p^e$ in $\Z^n$ is a polynomial in $p$. Is this true for \emph{subrings} in $\Z^n$ of index $p^e$? Let $f_n(k)$ denote the number of subrings of index $k$ in $\Z^n$. We can define the subring zeta function over $\Z^n$ to be $\zeta_{\Z^n}^R(s) = \sum_{k \ge 1} f_n(k)k^{-s}$. Is this zeta function uniform? These two questions are closely related.
	
	 In this paper, we describe what is known about these questions, and we make progress toward answering them in a couple ways. First, we describe the connection between counting subrings of index $p^e$ in $\Z^n$ and counting the solutions to a corresponding set of equations modulo various powers of $p$. We then show that  the number of solutions to certain subsets of these equations is a polynomial in $p$ for any fixed $n$. On the other hand, we give an example for which the number of solutions to a certain subset of equations is not polynomial. Finally, we give an explicit polynomial formula for the number of `irreducible' subrings of index $p^{n+2}$ in $\Z^n$.
\end{abstract}

\section{Introduction}
Throughout this paper, we will require subrings to contain the identity. Let $f_n(k)$ denote the number of subrings of index $k$ in $\Z^n$. Since this function is multiplicative, it suffices to consider $f_n(p^e)$ for each prime $p$. We are motivated by the main question below.

\begin{question}\label{poly_ques}
	For fixed integers $n$ and $e$, is $f_n(p^e)$ a polynomial in $p$?
\end{question}

While this question is interesting in its own right, understanding the answer gives us information about a related zeta function. Following the notation conventions of \cite{akkm, ish_subrings}, we define
\begin{align*}
\zeta_{\Z^n}^R(s) &= \sum_{\substack{S \text{ is a finite index}\\\text{subring  of } \Z^n}} [\Z^n: S]^{-s} \\
&= \sum_{k\ge 1 } f_n(k)k^{-s}.
\end{align*}
This zeta function is known to have an Euler product with local factors
$$
\zeta_{\Z^n, \, p}^R(s) = \sum_{e \ge 0} f_n(p^e) p^{-es}.
$$
Each local factor $\zeta_{\Z^n, \, p}^R(s)$ can be written as a rational function in $p$ and $p^{-s}$ \cite{dusautoy, gss}. It is not known how these rational functions vary with $p$. 

\begin{defn}
	A zeta function $\zeta^*(s)$ is \emph{finitely uniform} if there exist finitely many rational functions $W_1(X,Y)$ $, \ldots, W_r(X,Y) \in \Q(X,Y)$ so that for every prime $p$, $\zeta_p^*(s) = W_i(p, p^{-s})$ for some integer $i \in [1,r].$ If $r=1$, $\zeta^*(s)$ is called \emph{uniform}.
\end{defn}

\begin{question}\label{uniform_ques}
	Is $\zeta_{\Z^n}^R(s)$ uniform?
\end{question}

The zeta function $\zeta_{\Z^n}^R(s)$ is uniform if and only if for all $e$, $f_n(p^e)$ is a polynomial in $p$. Thus establishing an answer to Question \ref{poly_ques} immediately answers Question \ref{uniform_ques}.

There does not seem to be a consensus about what the answers to Questions \ref{poly_ques} or \ref{uniform_ques} should be. In this paper, we provide evidence that suggests an affirmative answer to both questions. We also discuss reasons why the answer could be no, and explain an algorithmic way of searching for counterexamples. 
We end by giving examples to show that we are currently at the boundary of what is computationally feasible and that new techniques may be needed to give a complete answer to Questions \ref{poly_ques} and \ref{uniform_ques}.

\subsection{Related Problems}
It is often too difficult to give an exact formula for a counting function. A step toward this problem is to classify the counting function. We will use three different categories: polynomial, quasipolynomial, and non-quasipolynomial. A function $F(k)$ is quasipolynomial in $k$ if there exist finitely many polynomials $G_1(k), \ldots, G_{N-1}(k)$ so that $F(k) = G_i(k)$ whenever $k\equiv i \pmod{N}.$ A function is non-quasipolynomial if it is not quasipolynomial.

Let $a_n(k)$ denote the number of subgroups of index $k$ in $\Z^n$. This counting function is well-understood.  In fact, there is an explicit formula for $a_n(p^e)$ that is polynomial in $p$, namely
\begin{equation}\label{subgrp_eq}
a_n(p^e) = \binom{n-1+e}{e}_p,
\end{equation}
where $\binom{}{}_p$ represents the $p$-binomial coefficient; see \cite[Chapter 1.8]{stanley} for a proof of this fact. We can define the subgroup zeta function for $\Z^n$ as
$$
\zeta_{\Z^n}^G(s) =  \sum_{k\ge 1} a_{n}(k)k^{-s}.
$$
The superscript $G$ is decoration to denote this is a zeta function of a group. This zeta function has an Euler product with local factors
$$
\zeta_{\Z^n, \, p}^G(s) = \sum_{e\ge 0} a_{n}(p^e)p^{-es}.
$$
Equation \eqref{subgrp_eq} implies that $\zeta_{\Z^n}^G(s)$ is uniform, and that
$$
\zeta_{\Z^n}^G(s) = \zeta(s) \zeta(s-1) \cdots \zeta(s-n+1).
$$
This fact is well-known, see Lubotzky and Segal \cite{lubotzkysegal} for five different proofs.

Questions \ref{poly_ques} and \ref{uniform_ques} are motivated by this example. Perhaps surprisingly, we find that the multiplicative structure of subrings in $\Z^n$ makes this problem much more difficult. In Sections \ref{prev_work} and \ref{background}, we will discuss how little is known about $f_n(k)$ and how difficult the problem turns out to be. Before moving to subrings in $\Z^n$, we highlight some examples of related subring and subgroup zeta functions to show that these zeta functions need not be uniform.

Let $\mathcal{O}_K$ be the ring of integers in a quadratic number field $K$. Set $b_K(p^e)$ to be the number of subrings in $\mathcal{O}_K$ of index $p^e$ \emph{with or without identity}. For a fixed quadratic number field $K$ and fixed exponent $e$, how does $b_K(p^e)$ vary as a function of $p$? Snocken \cite{snocken} studies the subring zeta function $\zeta_{\mathcal{O}_K}^{R, *}(s)$ over $\mathcal{O}_K$. This zeta function has an Euler product and its local factors are
$$\zeta_{\mathcal{O}_K, \, p}^{R, *}(s) = \sum_{e \ge 0} b_K(p^e) p^{-es}.$$ 
We use the notation $\zeta_{\mathcal{O}_K}^{R, *}(s)$ to represent that the subrings need not have identity.

He explicitly computes the local factor in the cases when $p$ is ramified, inert, or split in $\mathcal{O}_K$. 

\begin{prop} \cite[Proposition 7.20]{snocken}
	Let $\mathcal{O}_K$ be the ring of integers in a quadratic number field $K$. Then
	$$
	\zeta_{\mathcal{O}_K, \, p}^{R, *}(s)  = 
	\begin{cases}
		\frac{1-p^{-2s}}{(1-p^{-s})^2(1-p^{1-3s})} & \text{if } p \text{ is ramified in } \mathcal{O}_K\\[4pt]
			\frac{1-p^{-4s}}{(1-p^{-s})(1-p^{-2s})(1-p^{1-3s})} & \text{if } p \text{ is inert in } \mathcal{O}_K\\[4pt]
				\frac{1-p^{-2s}}{(1-p^{-s})^3(1-p^{1-3s})} &\text{if }  p \text{ is split in } \mathcal{O}_K.
	\end{cases}
	$$
\end{prop}
As a corollary, we see that $\zeta_{\mathcal{O}_K}^{R, *}(s)$ is finitely uniform. This implies that $b_K(p^e)$ is quasipolynomial in $p$.

The story for quadratic fields is quite different when restricting to finite index subrings in $\mathcal{O}_K$ that contain the identity, that is, orders in $\mathcal{O}_K$. When $K$ is a quadratic number field, there is a unique order of index $k$ for each positive integer $k$. Letting $o_K(k)$ be the number of orders in $\mathcal{O}_K$ of index $k$, we find $o_K(p^e) =1$ is a polynomial in $p$ for each fixed quadratic number field $K$ and integer $e \ge 0$. Further the corresponding zeta function is $$\zeta_{\mathcal{O}_K}^R(s) = \sum_{k\ge 1} o_K(k)k^{-s} = \zeta(s)$$ where $\zeta(s)$ is the Riemann zeta function. Thus $\zeta_{\mathcal{O}_K}^R(s)$ is uniform. 

 Kaplan, Marcinek, and Takloo-Bighash \cite[Proposition 1]{kmt} show that $o_K(p)$ depends on the decomposition of $p$ in $\mathcal{O}_K$. This implies that $\zeta_{\mathcal{O}_K}^R(s)$ will not be uniform when the degree of the number field is larger than 2. There is no conjecture about whether this zeta function should be finitely uniform. The zeta functions $\zeta_{\Z^n}^R(s)$ and $\zeta_{\mathcal{O}_K}^R(s)$ are closely related - if $p$ splits completely in $\mathcal{O}_K$, then $\zeta_{\Z^n ,p}^R(s) =\zeta_{\mathcal{O}_K,p}^R(s).$

There are examples of subgroup zeta functions that are not finitely uniform. See \cite[Theorem 2.20]{dusautoy} for an example using a group that is based on the $\F_q$-points of the elliptic curve $y^2 = x^3-x$. However, Voll \cite{voll} states that there are no known examples of subring zeta functions that are not finitely uniform.

\subsection{Finding non-polynomial pieces}
One way to show that certain zeta functions are not uniform is to show that for some fixed $e$, the sequence of coefficients $c(p^e)$  is not polynomial in $p$. It is often difficult to give an exact count for $c(p^e)$, so we split the count into several pieces.

For example, in order to count the number of $n\times n$ invertible matrices in a ring $R$, one must understand what the conditions on the entries are so that the determinant is a unit. Since there are too many entries to consider as $n$ grows large, one could fix a pattern of entries to be equal to 0 and consider the simpler determinant to find conditions on the remaining entries, then use Inclusion-Exclusion. Let $m_P(q)$ denote the number of invertible matrices with entries in $\mathbb{F}_q$ with all entries in pattern $P$ set to 0. There are several examples of $P$ for which $m_P(q)$ is not polynomial in $q$ \cite{stembridge}; however, it is well-known that the number of invertible $n\times n $ matrices with entries in $\F_q$ is equal to
$$|\text{GL}_n(\mathbb{F}_q)| = \prod_{i=0}^{n-1}(q^n-q^i)$$ and that this function is a polynomial in $q$ for each fixed $n$.

Thus it is possible to split the count for $c(p^e)$ into several pieces which are non-polynomial, yet have $c(p^e)$ be a polynomial.  This shows that finding one non-polynomial piece of a counting formula does not prove that $c(p^e)$ is non-polynomial or that the corresponding zeta function is not uniform. However, it is still a reasonable step to show that parts of the count are non-polynomial. This is often done in the literature; see \cite{paksoffer,vaughanlee}.

\subsection{Previous Work}\label{prev_work}
The subring zeta function $\zeta_{\Z^n}^R(s)$ can be written in terms of simpler zeta functions and Euler products when $n \le 4$. 
\begin{thm}\label{zeta_closed_form} \cite{datskovksywright, nakagawa}
Let $\zeta(s)$ denote the Riemann zeta function.	We have
	\begin{align*}
		\zeta_{\Z^2}^R(s) &= \zeta(s)\\
		\zeta_{\Z^3}^R(s) &= \frac{\zeta(3s-1)\zeta(s)^3}{\zeta(2s)^s}\\
		\zeta_{\Z^4}^R(s) &= \prod_p \frac{1}{(1-p^{-s})^2(1-p^{2-4s})(1-p^{3-6s})}\bigg( 1 + 4p^{-s}\\
		&+2p^{-2s} +(4p-3)p^{-3s} + (5p-1)p^{-4s} + (p^2-5p)p^{-5s} \\
		&+ (3p^2-4p)p^{-6s}-2p^{2-7s} - 4p^{2-8s} - p^{2-9s}\bigg).
	\end{align*}
\end{thm}
Theorem \ref{zeta_closed_form} implies that $\zeta_{\Z^n}^R(s)$ is uniform when $n \le 4$. There is not even a conjecture about the uniformity of $\zeta_{\Z^5}^R(s)$. 

It is unknown how to use current methods in the literature to give an exact expression for $\zeta_{\Z^n}^R(s)$ when $n \ge 5$. However, we can make progress toward understanding the uniformity of this zeta function by studying the coefficients $f_n(p^e)$. To do so, we break the count into several pieces by decomposing a subring into irreducible subrings.

\begin{defn}\cite{liu} A subring $S$ in $\Z^n$ is \emph{irreducible} if $(x_1, \ldots, x_n) \in S$ implies that $x_1 \equiv \cdots \equiv x_n \pmod{p}$.\end{defn}

 Liu \cite{liu} justifies his use of the term irreducible by showing that a subring in $\Z^n$ can be written as a direct sum of irreducible subrings. Liu also gives a recurrence relation as stated below. Let $g_n(p^e)$ denote the number of irreducible subrings of index $p^e$ in $\Z^n$. Note that Liu writes $g_{n+1}(p^e)$ to denote the number of irreducible subrings of index $p^e$ in $\Z^n$. We shift notation to match that of \cite{akkm, ish_subrings}.

\begin{prop} \cite[Proposition 4.4]{liu} \label{recurrence} Set $f_0(p^e) = 1$ if $e=0$ and $f_0(p^e)=0$ for $e >0$. We have
	$$f_n(p^e) = \sum_{i=0}^e \sum_{j=1}^n \binom{n-1}{ j-1}f_{n-j}(p^{e-i})g_j(p^i).$$
\end{prop}

 Liu's work suggests that it is sufficient to study irreducible subrings in $\Z^n$.

\begin{defn}
	A $n \times n$ matrix is in \emph{Hermite normal form} if it is upper triangular and $0 \le a_{ij} < a_{ii}$ for all $1\le i <j \le n$.
\end{defn}

There is a one-to-one correspondence between subrings of index $k$ in $\Z^n$ and $n \times n$ \emph{subring matrices} with determinant $k$. That is, matrices $A$ in Hermite normal form with determinant $k$ satisfying the following conditions:
\begin{enumerate}
	\item For each pair of columns  $v_i=(x_1, \ldots, x_n)^T$ and $v_j=(y_1, \ldots, y_n)^T$ in $A$, we have $v_i \circ v_j = (x_1y_1, \ldots, x_ny_n)^T$ is in the $\Z$-column span of $A$.
	\item We have $(1, \ldots, 1)^T$ is in the $\Z$-column span of $A$.
\end{enumerate}  
 This correspondence is described in detail in \cite{liu}.

Liu also establishes a one-to-one correspondence between irreducible subrings of index $p^e$ and \emph{irreducible $n \times n$ subring matrices} with determinant $p^e$, which are subring matrices with the additional property that the first $n-1$ columns are divisible by $p$ and $(1, \ldots, 1)^T$ is the last column. An irreducible subring matrix with determinant $p^e$ has the form
$$
A = \begin{pmatrix}
	p^{e_1} & pa_{12} & pa_{13} & \cdots & pa_{1(n-1)}&1 \\
	& p^{e_2} & pa_{23} & \cdots &pa_{2(n-1)} &1\\
	&& p^{e_3} & \cdots & pa_{3(n-1)} &1\\
	&&& \ddots & \vdots& \vdots\\
	&&&&p^{e_{n-1}} &1\\
	&&&&&1
\end{pmatrix}
$$
where $e_i > 0$ for each $1\le i< n$ and $e_1 + \ldots + e_{n-1} = e.$ Let $C_{n,e}$ denote the set of compositions of $e$ into $n-1$ parts. For each $\alpha = (e_1, \ldots, e_{n-1}) \in C_{n,e}$, let $g_\alpha(p)$ denote the number of irreducible subring matrices with given diagonal $(p^{e_1}, \ldots, p^{e_{n-1}}, 1)$. Observe that

$$
g_{n}(p^e) = \sum_{\alpha \in C_{n,e}} g_\alpha(p).
$$

Liu \cite{liu} and Atanasov, Kaplan, Krakoff, and Menzel \cite{akkm} give explicit polynomial formulas for $f_n(p^e)$ when $e \le 8$ and for $g_n(p^i)$ when $i \le n+1$. Atanasov et al. \cite{akkm} prove their results by studying $g_\alpha(p)$ for each diagonal $\alpha \in C_{n,e}$. This method will be described in more detail in Section \ref{background}.

\begin{prop}\cite[Proposition 4.3]{liu}
	Let $n > 0$. We have $g_n(p^i) = 0$ if $i < n-1$,  $g_n(p^{n-1}) =1 ,$ and $g_{n}(p^n) = \frac{p^{n-1}-1}{p-1}.$
	
\end{prop}

\begin{prop}\cite[Corollary 3.7]{akkm}\label{gn+1}
Let $n >0$. We have
\begin{align*}
	g_{n}(p^{n+1})&= \frac{1}{2(p-1)^2(p+1)}\bigg(2p^{2n-3} + (n^2-n) p^{n+1} \\
	&-(n^2-n)p^n -(n^2-n+2)p^{n-1} +(n^2-n-2) p^{n-2}+2\bigg)
\end{align*}	
For any fixed $n > 0$, this is a polynomial in $p$.
\end{prop}

\subsection{Main Results}
In this paper, we provide evidence that for each fixed $n$ and $e$, the function $f_n(p^e)$ may be polynomial in $p$. In fact, we consider something stronger, namely that $g_\alpha(p)$ may be polynomial for each $\alpha \in C_{n,e}$. 

First, we show that the number of solutions to one condition $v_i \circ v_j$ being in the $\Z$-column span of a matrix $A$ of the form
$$
A = \begin{pmatrix}
	p^{e_1} & pa_{12} & pa_{13} & \cdots & pa_{1(n-1)}&1 \\
	& p^{e_2} & pa_{23} & \cdots &pa_{2(n-1)} &1\\
	&& p^{e_3} & \cdots & pa_{3(n-1)} &1\\
	&&& \ddots & \vdots& \vdots\\
	&&&&p^{e_{n-1}} &1\\
	&&&&&1
\end{pmatrix}
$$
is polynomial. Recall that $g_\alpha(p)$ can be understood by counting  the set of solutions so that $v_i \circ v_j \in \text{Col}(A)$ for all $1 \le i \le j \le n$. Thus we will show that the set of solutions so that a single $v_i \circ v_j \in \text{Col}(A)$ is a polynomial in $p$. We will then provide an example that shows the number of solutions to $v_i \circ v_j, v_k \circ v_{\ell} \in \text{Col}(A)$ need not be polynomial. While this proves that pieces of the counting function $g_\alpha(p)$ can be non-polynomial, it does not mean $g_\alpha(p)$ is non-polynomial.

Second, we give an explicit formula for $g_n(p^{n+2})$, which is a polynomial in $p$; see Theorem \ref{n+2indexthm}. We do so by considering the set of all compositions $\alpha$ of $n+2$ into $n-1$ parts. The only possible $\alpha$ are those which are permutations of the multisets $\{4, 1, \ldots, 1\}$ , $\{3, 2, 1, \ldots, 1\}$, or $\{2, 2,2, 1, \ldots, 1\}$ with cardinality $n-1$. We show that for each such $\alpha$, $g_\alpha(p)$ is a polynomial in $p$. Some of these cases were previously known by Atanasov et al. \cite{akkm}. Even more, we show that if $\alpha$ is a permutation of the multiset $\{3, \beta, 1, \ldots, 1\}$ with cardinality $n-1$ and $\beta > 0$, then $g_\alpha(p)$ is polynomial in $p$. These results also give a better understanding of $g_n(p^e)$ for $e > n+2$; in particular, we find some pieces in the formula for $g_n(p^e)$ which are polynomial. 

This seems to suggest that if $g_\alpha(p)$ is not a polynomial in $p$, then the composition $\alpha$ is fairly complicated. In particular, $\alpha$ likely contains more than two entries that are not equal to 1. The more complicated $\alpha$ is, the less likely we will be able to use computational methods like those developed in \cite{ish_subrings} and summarized in Section \ref{background}.

While we show evidence that $f_n(p^e)$ may be polynomial, it is still likely that this function will not be. In Section \ref{background}, we show that we can compute $f_n(p^e)$ by computing the solutions to a system of polynomials modulo various powers of $p$. It is certainly possible that the set of solutions to these systems of polynomials fit within the philosophy of Mn\"ev's Universality Theorem \cite{mnev} or Vakil's Murphy's Law in algebraic geometry \cite{vakil}, and thus as $n$ and $e$ grow large, these solution sets could grow arbitrarily complicated. In Example \ref{maybe_qp_ex}, we show that some of the pieces involved in the counting function $f_n(p^e)$ may not be polynomial.

\section{Background}\label{background}
Atanasov et al. determine $f_n(p^e)$ for $e \le 8$ by explicitly writing down all of the multiplicative closure conditions for irreducible subring matrices with fixed diagonal $(p^{e_1}, \ldots, p^{e_{n-1}}, 1)$ such that $e_1 + \ldots + e_{n-1} \le 8.$ They determine $g_n(p^{n+1})$ similarly by considering fixed diagonals that come from compositions of $n+1$ into $n-1$ parts. This is often tedious and does not scale well as $e$ and $n$ grow. In previous work \cite{ish_subrings}, we establish a new technique for understanding the closure conditions via linear algebra. We show that $v_i \circ v_j \in \text{Col}(A)$ if and only if $A \vec{x}^T= v_i \circ v_j$ for some $\vec{x} \in \Z^n$. Thus we can understand the condition $v_i \circ v_j \in \text{Col}(A)$ by setting up the augmented matrix $[A \, v_i \circ v_j]$ and row reducing. It is natural to assume the row reduction will be complicated since we are row reducing an integer matrix over $\Z$; however we demonstrate that after dividing each row by $p^{e_i}$, the remaining row reduction steps amount to adding an integer multiple of one row to another. Thus in order to count $g_\alpha(p)$, we must count the number of simultaneous solutions so that the last column of the echelon form of $[A \; v_i \circ v_j]$ (considered over $\Q$) contains only integer entries. This is equivalent to counting the number of simultaneous solutions to the vanishing of several polynomials modulo powers of $p$. See \cite[Section 3]{ish_subrings} for more details. 
We summarize this in the following proposition.

\begin{prop}
	Fix $\alpha=(e_1, \ldots, e_{n-1})$ and let $0\le a_{ij}< p^{e_{i}}$ for each $1\le i < j <n$. Let $\vec{a}_{ij} = (a_{12},a_{13}, \ldots, a_{1(n-1)}, a_{23}, \ldots, a_{(n-2)(n-1)})$. There exist finitely many polynomials $h_1(\vec{a}_{ij}), \ldots, h_{\ell}(\vec{a}_{ij})$ and finitely many positive integers $r_1, \ldots, r_\ell$ so that $g_\alpha(p)$ is equal to the number of solutions to the system $h_1(\vec{a}_{ij}) \equiv 0 \pmod{p^{r_1}}, \ldots, h_\ell(\vec{a}_{ij}) \equiv 0 \pmod{p^{r_\ell}}$.
\end{prop}
\begin{example}
	Consider the diagonal $\alpha = (3,2)$. We set up the matrix
	$$
	\begin{pmatrix}
		p^3 & pa_{12} & 1\\
		&p^2&1\\
		&&1
	\end{pmatrix}
	$$
	where $a_{12} \in [0, p^2)$. We must determine conditions on the variable $a_{12}$ so that this matrix is an irreducible subring matrix. We do so by computing $v_i \circ v_j$ for each $1 \le i \le j \le 3$ and asking what conditions there are on $a_{12}$ so that $v_i \circ v_j \in \text{Col}(A)$.  
	
	For each $1\le i \le 3$, we can write $v_i \circ v_3  = v_3$, so $v_i \circ v_3\in \text{Col}(A)$. When we compute $v_1 \circ v_1$ and $ v_1 \circ v_2$ and solve $A\vec{x}= v_1 \circ v_1$ or $A \vec{y} = v_1 \circ v_2$, we find that the solutions $\vec{x}$ and $\vec{y}$ are vectors in $\Z^3$. Since these are vectors with integral entries, then $v_1 \circ v_1$ and $v_1 \circ v_2$ must be in the $\Z$-column span of $A$.
	
	 The only remaining product to consider is $v_2 \circ v_2.$  Following the method summarized above, we set up the augmented matrix
		$$
	\begin{pmatrix}
		p^3 & pa_{12} & p^2a_{12}^2\\
		&p^2&p^4\\
		&&0
	\end{pmatrix}.
	$$
	Notice that we can omit the $(1,1,1)^T$ column since $v_2 \circ v_2$ has a 0 in the third entry. Then we row reduce, obtaining the matrix in echelon form (over $\Q$)
	$$
	\begin{pmatrix}
		1 &0 &  \frac{a_{12}^2}{p}- pa_{12}\\
		&1&p^2\\
		&&0
	\end{pmatrix}.
	$$
  
	Thus $v_2 \circ v_2 \in \text{Col}(A)$ if and only if $\frac{a_{12}^2}{p} \in\Z$. This gives the condition $a_{12}^2 \equiv 0 \pmod{p}$, so we require $a_{12} = mp$ for any $0 \le m < p$. Thus $g_\alpha(p) =  p.$
\end{example}

\section{There are a polynomial number of solutions to a single $v_i \circ v_j \in \text{Col}(A)$}
\begin{prop}\label{vivi_poly}
	Consider 	$$
	A = \begin{pmatrix}
		p^{e_1} & pa_{12} & pa_{13} & \cdots & pa_{1(n-1)} &1\\
		& p^{e_2} & pa_{23} & \cdots &pa_{2(n-1)}&1 \\
		&& p^{e_3} & \cdots & pa_{3(n-1)}&1 \\
		&&& \ddots & & \vdots\\
		&&&&p^{e_{n-1}}&1 \\
		&&&&&1
	\end{pmatrix}
	$$
	where $e_i >0$ for each $1 \le i<n$ and the columns are labeled $v_1, \ldots, v_n$. There are a polynomial number of solutions $\{a_{rs} \, :\, 1 \le r < s\le i \}$ so that $v_i \circ v_i \in \text{Col}(A)$.
\end{prop}
\begin{proof}
	Take some column $v_i$ and consider $v_i \circ v_i$. Since column $v_j$ has a 0 in the $i^{th}$ entry whenever $j > i$, we need only consider the simpler matrix
	$$\begin{pmatrix}
		p^{e_1} & pa_{12} & pa_{13} & \cdots & pa_{1i} &p^2a_{1i}^2\\
		& p^{e_2} & pa_{23} & \cdots &pa_{2i}&p^2a_{2i}^2 \\
		&& p^{e_3} & \cdots & pa_{3i}&p^2a_{3i}^2\\
		&&& \ddots & & \vdots\\
		&&&&p^{e_i}&p^{2e_i} 
	\end{pmatrix}.
	$$
	 We use the row reduction method to understand solutions to $v_i \circ v_i \in \text{Col}(A)$. The idea is that $v_i \circ v_i \in \text{Col}(A)$ if and only if the last column of this simpler matrix has integer coefficients after row reducing.
	 
	 Notice that the condition coming from the $(i-1)^{th}$ entry in the last column depends only on $a_{(i-1)i}$. So we will solve first for variable $a_{(i-1)i}$. Then after fixing a choice of $a_{(i-1)i}$, the $(i-2)^{th}$ entry in the last column only depends on $a_{(i-2)i}$ and $a_{(i-2)(i-1)}$. So we can find restrictions on $a_{(i-2)i}$ and $a_{(i-2)(i-1)}$ using this second condition. The main idea is that at each step, we can substitute in our choice for the previous variables and then solve for a small amount of new variables.  \\
	
	The first row reduction step shows that $v_i \circ v_i \in \text{Col}(A)$ implies
	$$
	D_{i-1} = \frac{p^2a_{(i-1)i}^2-p^{e_i+1}a_{(i-1)i}}{p^{e_{i-1}}} \in \Z.
	$$
	Note that for each fixed $p$, $D_{i-1}$ is a function of the variable $a_{(i-1)i}$, but we are suppressing this dependence for clarity.
	
	We want to show that $D_{i-1} \in \Z$ has a polynomial number of solutions for the variable $a_{(i-1)i}$. However, we will need a slightly stronger result to continue this argument. Namely, we need to show that for each $\ell$, $D_{i-1} \in p^{\ell} \Z$ has a polynomial number of solutions. We will explain why this strengthened condition is necessary later on.\\
	
	\noindent \textbf{Claim:} Fix diagonal $(e_1, \ldots, e_{n-1})$ and integer $\ell \ge 0$. The expression $D_{i-1} \in p^{\ell} \Z$ has a polynomial number of solutions for $a_{(i-1)i}.$\\
	\begin{proof}[Proof of Claim]
		We can rewrite $D_{i-1} \in p^{\ell}\Z$ as 
		$$	\frac{p^2a_{(i-1)i}^2-p^{e_{i}+1}a_{(i-1)i}}{p^{e_{i-1}+\ell}} \in \Z.$$
		
		If $e_{i-1}+\ell \le 2$, this condition holds for all possible choices of $a_{(i-1)i}$. Otherwise, we can simplify this to \begin{equation}\label{cond1} p^{e_{i-1}+\ell-2} \mid \left(a_{(i-1)i}^2 - p^{e_i-1}a_{(i-1)i}\right).
		\end{equation}
		
		If $e_i=1$, we get $p^{e_{i-1}+\ell-2} \mid a_{(i-1)i}(a_{(i-1)i}-1)$. There are a polynomial number of solutions in this case. 
		
		Now suppose $e_i >1$. Observe that \eqref{cond1} implies $p \mid a_{(i-1)i}$. Write $a_{(i-1)i} = pd$ where $d$ must be in $[0, p^{e_{i-1}-2})$. If $e_{i-1}+\ell \le 4$, the condition holds for all choices of $d \in [0, p^{e_{i-1}-2})$. If $e_i = 2$, then we solve $p^{e_{i-1}+\ell-4} \mid d(d-1)$, which has a polynomial number of solutions. Otherwise $p^{e_{i-1}+\ell-4} \mid (d^2 - p^{e_i-2}d)$ implies $p \mid d$. We can continue by setting $d = pd^\prime$ and showing that either the the denominator divides the numerator evenly, $p^{e_{i-1}+\ell-6} \mid d^\prime (d^\prime -1)$, or $p \mid d^\prime$. 
		
		We continue until we are in one of the first two cases or until we find $p^{e_{i-1}} \mid a_{(i-1)i}$ implying that $a_{(i-1)i} =0$. In either case, we find a polynomial number of values for $a_{(i-1)i}$ that satisfy $D_{i-1} \in p^\ell \Z$. Notice that the value of $e_{i-1}, e_i$, and $\ell$ determine when this process stops, so for each fixed $e_{i-1}, e_i,$ and $\ell$, there are a polynomial number of choices for $a_{(i-1)i} \in [0, p^{e_{i-1}-1}).$
	\end{proof}

	\noindent The next row reduction step shows that $v_i \circ v_i \in \text{Col}(A)$ implies
	$$
	D_{i-2} = \frac{p^2a_{(i-2)i}^2- p^{e_{i-1}+1}a_{(i-2)i} - pa_{(i-2)(i-1)}D_{i-1}}{p^{e_{i-2}}} \in \Z.
	$$
Let $\nu_p(x)$ be the largest non-negative integer $k$ so that $p^k \mid x$. Fix some choice of $a_{(i-1)i}$ satisfying the $D_{i-1}(a_{(i-1)i}) \in \Z$ and let $D_{i-1}$ denote $D_{i-1}(a_{(i-1)i})$. 

If $D_{i-1} = 0$, we can solve for $a_{(i-2)i}$ as in the first step.  Otherwise, write $\ell = \nu_p(D_{i-1})$ and set $D_{i-1} = p^{\ell} d$. Notice that since we are fixing the valuation of $D_{i-1}$, we need to know that there are a polynomial number of solutions to $D_{i-1} \in \Z$ and $v_p(D_{i-1})  = \ell$, hence our stronger claim above.
	
	As before, we need to prove a slightly stronger claim.\\
	
	\noindent \textbf{Claim:} Fix diagonal $(e_1, \ldots e_{n-1})$ and integer $m \ge 0$. The condition $D_{i-2} \in p^m \Z$ has a polynomial number of solutions for $\{a_{(i-2)i}, a_{(i-2)(i-1)}\}$.
	\begin{proof}[Proof of Claim]
		The condition $D_{i-2} \in p^m\Z$ is the same as
		$$
		D_{i-2} = \frac{p^2a_{(i-2)i}^2- p^{e_{i-1}+1}a_{(i-2)i} - pa_{(i-2)(i-1)}D_{i-1}}{p^{e_{i-2}+m}} \in \Z
		$$
		Assume $e_{i-2} +m> 1$ as otherwise, this condition is trivial.
		
		Recall that $v_p(D_{i-1}) = \ell. $ If $\ell= 0$, the condition simplifies to
		$$
		D_{i-2} = \frac{pa_{(i-2)i}^2- p^{e_{i-1}}a_{(i-2)i} - a_{(i-2)(i-1)}d}{p^{e_{i-2}+m-1}} \in \Z.
		$$
		In this case, note that $$a_{(i-2)(i-1)} \equiv d^{-1} (pa_{(i-2)i}^2 - p^{e_{i-1}}a_{(i-2)i}) \pmod{p^{e_{i-2}+m-1}}.$$ 
		
		If $\ell > 0$ and $e_{i-2}+m =2,$ the condition is trivial. Otherwise, $\ell > 0, e_{i-2}+m>2$, and the condition simplifies to
	\begin{equation}\label{cond2}
		D_{i-2} = \frac{a_{(i-2)i}^2- p^{e_{i-1}-1}a_{(i-2)i} - p^{\ell-1}a_{(i-2)(i-1)}d}{p^{e_{i-2}+m-2}} \in \Z.
		\end{equation}
		Write $f(a_{(i-2)i}) = a_{(i-2)i}^2- p^{e_{i-1}-1}a_{(i-2)i}$. If $e_{i-2}+m \le \ell+1$, then we simply solve $p^{e_{i-2}+m-2} \mid f(a_{(i-2)i})$, which is similar to before. Otherwise, notice that \eqref{cond2} implies $p^{\ell-1} \mid f(a_{(i-2)i})$. Write $\tilde{f}(a_{(i-2)i}) p^{\ell-1} = f(a_{(i-2)i})$. Then we simply solve
		$$
		\tilde{f}(a_{(i-2)i}) \equiv a_{(i-2)(i-1)}d \pmod{p^{e_{i-2}+m-\ell-1}}.
		$$
		Since $(d,p) =1$, we can multiply by $d^{-1}$, and solve for the variable	$a_{(i-2)(i-1)}$. 
		
		In any of these cases, we find a polynomial number of solutions to $D_{i-2} \in p^m\Z$.
	\end{proof}
	
	Continue inductively. The general step for $j < i$ is of the form
	$$
	D_{i-j} = \frac{p^2a_{(i-j)i}^2- p^{e_{i-j+1}+1}a_{(i-j)i} - p\sum_{k=1}^{j-1} a_{(i-j)(i-k)}D_{i-k}}{p^{e_{i-j}}} \in \Z
	$$
	
	For each fixed diagonal $(e_1, \ldots, e_{n-1})$ and integer $  m \ge 0$, we must show that $D_{i-j} \in p^m\Z$ has a polynomial number of solutions for $\{a_{(i-j)(i-k)}\, :\, 0 \le k < j\}$.
	
	Suppose that $D_{i-r} \ne 0$ for some $r \in [1,j-1]$. Let $\ell = \nu_p(D_{i-r})$ and write $D_{i-r} = p^\ell d$. 
	
	As before, there are different cases when $\ell =0$ and when $\ell > 0$. 
	
	If $\ell =0$, the condition reduces to 
	$$
	D_{i-j} = \frac{pa_{(i-j)i}^2- p^{e_{i-j+1}}a_{(i-j)i} -  a_{(i-j)(i-r)} d - \sum_{k \ne r} a_{(i-j)(i-k)}D_{i-k}}{p^{e_{i-j}+m-1}} \in \Z
	$$
	As before, we can simply solve for $a_{(i-j)(i-r)}$.\\
	
	If $\ell > 0$ and $e_{i-j}+m =1$, the condition is trivial. Otherwise, $\ell > 0$ and $e_{i-j}+m > 1$. Rewriting the condition gives
	$$
	D_{i-j} = \frac{pa_{(i-j)i}^2- p^{e_{i-j}}a_{(i-j)i} - p^{\ell} a_{(i-j)(i-r)} d - \sum_{k \ne r} a_{(i-j)(i-k)}D_{i-k}}{p^{e_{i-j}+m-1}} \in \Z
	$$
	
	If $1+ \ell \ge e_{i-j}+m$, then we simply solve the rest by induction. Otherwise, write
	$$
	f =p a_{(i-j)i}^2-  p^{e_{i-j}}a_{(i-j)i}  -\sum_{k \ne r} a_{(i-j)(i-k)}D_{i-k}
	$$
	and notice that $p^{\ell} \mid f$. Thus we can write $\tilde{f} p^{\ell}= f$ and solve
	$$
	a_{(i-j)(i-r)}  \equiv d^{-1}\tilde{f} \pmod{p^{e_{i-j}+m-1-\ell}}.
	$$
	There are a polynomial number of solutions in this case.\\
	
	Otherwise, all $D_{i-k} =0$, so we simply need to solve
	$$
	\frac{p^2a_{(i-j)i}^2- p^{e_{i-j+1}}a_{(i-j)i}}{p^{e_{i-j}}} \in p^m \Z,
	$$
	which has a polynomial number of solutions as seen in the case $D_{i-1} \in p^m\Z$.
\end{proof}

\vspace{.05in}
\begin{prop}
Let $i < j$ and consider $$
A = \begin{pmatrix}
	p^{e_1} & pa_{12} & pa_{13} & \cdots & pa_{1(n-1)} &1\\
	& p^{e_2} & pa_{23} & \cdots &pa_{2(n-1)}&1 \\
	&& p^{e_3} & \cdots & pa_{3(n-1)}&1 \\
	&&& \ddots & & \vdots\\
	&&&&p^{e_{n-1}}&1 \\
	&&&&&1
\end{pmatrix}
$$
where $e_i >0$ for each $1 \le i <n$ and the columns are labeled $v_1, \ldots, v_n$. There are a polynomial number of solutions $\{a_{rs}\,: \, 1\le r < s \le i\} \cup \{a_{rj} \, : \, 1 \le r \le i\}$ so that $v_i \circ v_j \in \text{Col}(A)$.
\end{prop}
\begin{proof}
	Take a pair of columns $v_i, v_j$ with $i < j$ and consider $v_i \circ v_j$. We can set up the simpler augmented matrix for the same reasons given in Proposition \ref{vivi_poly}. The simpler matrix is
	$$\begin{pmatrix}
		p^{e_1} & pa_{12} & pa_{13} & \cdots & pa_{1i} &p^2a_{1i}a_{1j}\\
		& p^{e_2} & pa_{23} & \cdots &pa_{2i}&p^2a_{2i}a_{2j} \\
		&& p^{e_3} & \cdots & pa_{3i}&p^2a_{3i}a_{3j}\\
		&&& \ddots & & \vdots\\
		&&&&p^{e_i}&p^{e_i+1}a_{ij}
	\end{pmatrix}.
	$$
As before, we consider the first row reduction step, which gives a condition that only depends on a few variables. We find
	$$
	D_{i-1} = \frac{p^2a_{(i-1)i} a_{(i-1)j} - p^2a_{(i-1)i}a_{ij}}{p^{e_{i-1}}} \in \Z
	$$
	which simplifies to
	\begin{equation} \label{ij_cond1}
	\frac{a_{(i-1)i} a_{(i-1)j} -a_{(i-1)i}a_{ij}}{p^{e_{i-1}-2}} \in \Z
	\end{equation}
	We note that $D_{i-1}$ depends on the variables $a_{(i-1)i}, a_{(i-1)j},$ and $ a_{ij}$  but we suppress this in the notation for clarity.
	
	As in Proposition \ref{vivi_poly}, we need to show that for each fixed diagonal $(e_1, \ldots , e_{n-1})$ and integer $\ell \ge 0$, there are a polynomial number of solutions to $D_{i-1} \in p^\ell \Z$.
	
	This is clearly true for all choices of variables if $e_{i-1}+\ell \le 2$. Otherwise, observe that \eqref{ij_cond1} and $D_{i-1} \in p^\ell \Z$ imply that
	$$
	a_{(i-1)i}(a_{(i-1)j} - a_{ij}) \equiv 0 \pmod{p^{e_{i-1}-2+\ell}}.
	$$
	Thus there exists positive integers $b,c$ so that $b+c= e_{i-1} -2+\ell$, $p^b\mid a_{(i-1)i}$, and $p^c\mid (a_{(i-1)j} - a_{ij})$. There are clearly a polynomial number of choices for the variables $a_{(i-1)i}, a_{(i-1)j}$, and $a_{ij}$.\\
	
	Fix a choice of  $a_{(i-1)i}, a_{(i-1)j}$, and $a_{ij}$ satisfying $D_{i-1} \in p^\ell \Z$. We compute the second row reduction step, and note that $v_i \circ v_j \in \text{Col}(A)$ implies
	$$
	D_{i-2}  = \frac{p^2a_{(i-2)i} a_{(i-2)j} - p^2a_{(i-2)i}a_{ij} - pa_{(i-2)(i-1)} D_{i-1}}{p^{e_{i-2}}} \in \Z.
	$$
	Notice that $a_{ij}$ and $D_{i-1}$ are now fixed integers.	Write $\ell = \nu_p(D_{i-1})$ and set $D_{i-1} = p^\ell d$. Then we can simplify the condition to
	$$
	D_{i-2}  = \frac{p^2a_{(i-2)i} a_{(i-2)j} - p^2a_{(i-2)i}a_{ij} - p^{\ell+1}a_{(i-2)(i-1)}d}{p^{e_{i-2}}} \in \Z.
	$$
	
	As before, we want to show that for each fixed diagonal $(e_1, \ldots, e_{n-1})$ and integer $\ell \ge 0$, $D_{i-2} \in p^m \Z$ has a polynomial number of solutions for $\{a_{(i-2)i}, a_{(i-2)j}, a_{(i-2)(i-1)}\}$. Thus we need to show that
	\begin{equation} \label{ij_cond2} \frac{p^2a_{(i-2)i} a_{(i-2)j} - p^2a_{(i-2)i}a_{ij} - p^{\ell+1}a_{(i-2)(i-1)}d}{p^{e_{i-2}+m}} \in \Z
	\end{equation}
	for all $m \ge 0$.

	If $\ell =0$ and $e_{i-2} +m\le 1$, then \eqref{ij_cond2} holds for all possible choices of the variables $a_{(i-2)i}, a_{(i-2)j}, a_{(i-2)(i-1)}$. If $\ell \ge 1$ and $e_{i-2}+m \le 2$, then \eqref{ij_cond2} holds for all possible choices of these variables.\\
	
	Assume $\ell =0$ and $e_{i-1} +m \ge 2$. Then we can simplify \eqref{ij_cond2} to
	$$
	D_{i-2}  = \frac{pa_{(i-2)i} a_{(i-2)j} - pa_{(i-2)i}a_{ij} -a_{(i-2)(i-1)}d}{p^{e_{i-2}+m-1}} \in \Z
	$$
	and then solve 
	$$
	a_{(i-2)(i-1)} \equiv d^{-1} \left(pa_{(i-2)i} a_{(i-2)j} - pa_{(i-2)i}a_{ij} \right) \pmod{p^{e_{i-2}+m-1}}.
	$$
	There are a polynomial number of solutions for all variables (recall that $a_{ij}$ is fixed).\\
	
	Assume $\ell \ge 1$ and $e_{i-1}+m \ge 3$. The condition simplifies to
	$$
	D_{i-2}  = \frac{a_{(i-2)i} a_{(i-2)j} - a_{(i-2)i}a_{ij} - p^{\ell-1}a_{(i-2)(i-1)} d}{p^{e_{i-2}+m-2}} \in \Z.
	$$
	
	If $\ell-1 \ge e_{i-2}+m-2$, then the condition simplifies to 
	$$
	\frac{a_{(i-2)i} a_{(i-2)j} - a_{(i-2)i}a_{ij}}{p^{e_{i-2}+m-2}} \in \Z.
	$$
	Recall that $a_{ij}$ is fixed and notice that we can solve this system similarly to step 1.\\
	
	Otherwise, set $f = a_{(i-2)i} a_{(i-2)j} - a_{(i-2)i}a_{ij}$. We see that $p^{\ell-1} \mid f$, so we can rewrite $f = p^{\ell-1} \tilde{f}$. The condition simplifies to
	$$
	\frac{\tilde{f} - a_{(i-2)(i-1)} d}{p^{e_{i-2}+m-\ell-1}} \in \Z.
	$$
	We can solve for 
	$$
	a_{(i-2)(i-1)} \equiv d^{-1} \tilde{f}\pmod{p^{e_{i-2}+m-\ell-1}}.
	$$
	
	A similar inductive proof goes through here.
\end{proof}

\begin{example}\label{maybe_qp_ex}
	While the number of solutions to $v_i \circ  v_j  \in \text{Col}(A)$ is polynomial in $p$ for any pair $(i,j)$ such that $1 \le i \le j  \le n$, the number of solutions to a pair $v_i \circ v_j, v_k \circ v_{\ell} \in \text{Col}(A)$ need not be. For example, let $\alpha = (3,2,2,2).$ We can compute $g_\alpha(p)$ by counting the simultaneous solutions to
	\begin{align}
		&-a_1a_4^2 + a_2^2\equiv 0 \pmod{p}\label{ex1a}\\
		&-a_1a_4a_5 + a_1a_4a_6 + a_2a_3 - a_2a_6 \equiv 0 \pmod{p}\label{ex1b}\\
		 &a_1a_4a_6^2 - a_1a_5^2p - a_1a_4a_6p - a_2a_6^2 + a_3^2p + a_2a_6p \equiv 0 \pmod{p^2}\label{ex1c}\\
		 &-a_4a_6^2 \equiv 0 \pmod{p}\label{ex1d}
	\end{align}
	Equation \eqref{ex1a} comes from $v_3 \circ v_3 \in \text{Col}(A)$, \eqref{ex1b} comes from $v_3 \circ v_4 \in \text{Col}(A)$, and \eqref{ex1c} - \eqref{ex1d} come from $v_4 \circ v_4 \in \text{Col}(A)$. The conditions $v_i \circ v_j \in \text{Col}(A)$ for $(i,j) \not \in\{ (3,3), (3,4), (4,4)\}$ do not give any additional restrictions on $a_{12}$.
	
	The number of solutions to the simultaneous conditions $v_3\circ v_3  \in  \text{Col}(A)$ and  $v_4 \circ v_4 \in \text{Col}(A)$ is quasipolynomial. In particular, the number of solutions to $v_3 \circ v_3 \in \text{Col}(A)$ and $v_4 \circ v_4 \in \text{Col}(A)$  is a polynomial in $p$ plus a polynomial multiple of $\#V(-a_1 a_4^2 + a_2^2, -a_1a_5^2 + a_3^2)$ where $V \subset \mb{A}^5(\F_p).$ We find
	$$\#V(-a_1 a_4^2 + a_2^2, -a_1a_5^2 + a_3^2) = \begin{cases}
		p^3 & p=2\\
		2p^3 - 3p^2 + 3p -1 & p>2
	\end{cases}.$$
	
	However, if we consider the simultaneous conditions $v_3\circ v_3\in \text{Col}(A),  v_3\circ v_4\in \text{Col}(A),$ and $v_4\circ v_4\in \text{Col}(A)$, there are a polynomial number of solutions. This implies that $g_\alpha(p)$ is polynomial in $p$. We still do not have an example of a diagonal $\alpha$ for which $g_\alpha(p)$ is not polynomial.
\end{example}

\section{Counting irreducible subrings of fixed index}
 The goal of this section is to show that for fixed $n>2$, $g_n(p^{n+2})$ is a polynomial in $p$. This result makes progress toward understanding the behavior of $g_n(p^e)$ for $e \ge n+2$ as a function in $p$. Further, Lemmas   \ref{lemma2beta} and \ref{lemma3beta} give partial information about $f_n(p^e)$ for each fixed $n> 2$ and $e \ge n+2$.

\begin{thm} \label{n+2indexthm} Let $n > 2$. The number of $n \times n$ irreducible subring matrices of index $p^{n+2}$ is a polynomial in $p$. In particular,\\
	\begin{align*}
		g_n(p^{n+2}) &= \frac{p^{n-5}}{24 (p-1 )^2}\bigg( 24 p^{2n-5}- 24 p^{2n-6}+ 24 p^{2 n-7}+ 12n(n-1) p^{n+1}\\
		&  - 12n(n-1) p^n - 	24 p^{n-2}- 12 (n-2)(n-3) p^{n-3}+ 12 (n-1)(n-4) p^{n-4} \\
		&+n (n-1)(n-2)(n-3) p^6   - 4n (n-1)^2(n-5)  p^5 \\
		& + (7 n^4-50n^3+41n^2+98n -120) p^4- 4 (n-2) (n+7) (n^2-7n + 9) p^3\\
		&+ (- 5 n^4+ 102 n^3 - 619 n^2 + 1482 n -1200 ) p^2 \\
		&+ 	4 (n-2) (n-3) (n-4) (2n-15) p	-(n-4)(n-5)(n-6)(3n-5)   \bigg).
	\end{align*}
\end{thm}

Theorem \ref{n+2indexthm} is consistent with the formulas for $g_{n}(p^{n+2})$ when $3 \le n \le 6$ given in \cite[Tables 1 and 2]{akkm}.


\subsection{The formula for $g_n(p^{n+2})$}
The main result is a formula for the number of irreducible $n \times n$ subring matrices of index $p^{n+2}$. 

Recall that $C_{n,e}$ denotes the set of compositions of $e$ into $n-1$ parts. Let $n > 2$. Observe that $C_{n, n+2}$ consists only of compositions that are permutations of the following multisets of cardinality $n-1$: $\{4, 1, \ldots, 1\}$ , $\{3, 2, 1, \ldots, 1\}$, or $\{2, 2,2, 1, \ldots, 1\}$. Therefore in order to understand $g_n(p^{n+2})$, it suffices to study $g_\alpha(p)$ for the compositions $\alpha$ of $n-1$ which are permutations of the multisets listed above.
Atanasov et al. \cite{akkm} handle several of these cases. 
\vspace{.05in}
\begin{lemma} \cite[Lemma 3.4]{akkm}
	Let $\alpha = (1, \alpha_2, \ldots, \alpha_k)$ be a composition of positive integer $e$ and let $\alpha' = (\alpha_2, \ldots \alpha_k)$. Then $g_\alpha(p) = g_{\alpha'}(p)$. 
\end{lemma}
\vspace{.05in}
\begin{lemma}\cite[Lemma 3.5, $\beta=4$]{akkm} \label{lemma4}
	Let $\alpha = (4, 1, \ldots, 1)$ be a composition of length $n-1$. Then 
	$$
	g_\alpha(p) =(n-1)p^{n-2}.
	$$
\end{lemma}
\vspace{.05in}

\begin{lemma}\cite[Lemma 3.6]{akkm}\label{lemma2beta}
	Let $\alpha = (2, 1, \ldots,1,  \beta, 1 , \ldots, 1)$ be a composition of length $n-1$ with $\beta > 1$ in the $k^{th}$ position. Then
	\begin{enumerate}
		\item If $\beta = 2$, $g_\alpha(p) = p^{2n-k-4} + (n-k) p^{n-3}(p-1)$.
		\item If $\beta \geq 3$, $g_\alpha(p) =(n-k) \left(p^{2n-k-4} +p^{n-3}(p-1)\right).$
	\end{enumerate}
\end{lemma}

We can derive the following corollary of Lemma \ref{lemma2beta} by summing over all permutations of the diagonal $\alpha = (2, 1, \ldots, 1,\beta, 1, \ldots, 1).$
\vspace{.05in}
\begin{cor}\label{cor2beta} Let $\alpha = (2, 1, \ldots, \beta, 1 , \ldots, 1)$ be a composition of length $n-1$ with $\beta > 2$ in the $k^{th}$ position for some $k \in [2, n-1]$. Let $S_{2, \beta}$ be the set of all such $\alpha$. Then
	$$\sum_{\alpha \in S_{2, \beta}} g_\alpha(p) = \frac{p^{n- 3} ((n - 2) p^{n - 1} - (n - 1) p^{n-2} + 1))}{(p - 1)^2}+ (p-1)p^{n-3} \binom{n-1}{2}. $$
\end{cor}
\begin{proof}
	We consider the sum term by term.\\
	
	\noindent \textbf{Term 1:}
	\begin{align*}
		\sum_{k=2}^{n-1} (n-k) p^{2n-k-4} &=  \frac{p^{n- 3} ((n - 2) p^{n - 1} - (n - 1) p^{n-2} + 1))}{(p - 1)^2}.
	\end{align*}
	
	\noindent \textbf{Term 2:}
	\begin{align*}
		\sum_{k=2}^{n-1} (n-k) p^{n-3} (p-1) &= (p-1)p^{n-3} \sum_{k=2}^{n-1} (n-k)\\
		&= (p-1)p^{n-3}\binom{n-1}{2}.
	\end{align*}
\end{proof}

There are two main cases remaining. First, we consider $\alpha$ to be a permutation of the multiset $\{3, 2, 1\ldots, 1\}$ for which the first entry is 3. Second, we consider $\alpha$ to be a permutation of the multiset $\{2, 2,2,1, \ldots, 1\}$ for which the first entry is 2. We defer proofs of the following lemmas and corollaries until the appendix as the proofs are fairly long. We emphasize that Lemma \ref{lemma3beta} gives partial results about $g_n(p^e)$ for $e > n+2$ as well. These results are consistent with the computations given in \cite[Tables 3, 4, 5]{akkm}.
\vspace{.05in}
\begin{lemma} \label{lemma3beta}
	Let $\alpha = (3, 1, \ldots, \beta, 1, \ldots, 1)$ be a composition of length $n-1$ with $\beta > 1$ in the $(k+1)$ position. Then 
	\begin{enumerate}
		\item If $\beta =2$, 
		\begin{align*}
			g_\alpha(p) &= (n-1)p^{2n-k-5} +   (n-k-1)(n-2) p^{n-3} (p-1)- (n-k-1)p^{n-3} \\
			&+\left((n-k-2) + \binom{n-k-2}{2}\right)p^{n-3}(p-2)(p-3) .
		\end{align*}
		\item If $\beta > 2$, 
		\begin{align*}
			g_\alpha(p) &= (n-k-1)(n-2) p^{2n-k-5} \\
			&+ \left((n-k-1)(k-1) + 1 + (n-k-2)(n-k-3)\right)  p^{n-3} (p-1) \\
			&+(n-k-2)p^{n-3}(p-1)^2 + (n-k-2)p^{n-2} (p-1)\\
			&+ (n-k-2)(n-k-3)p^{n-3}(p-1)(p-2).
		\end{align*}
	\end{enumerate}
\end{lemma}
\vspace{.05in}

\begin{cor} \label{cor32}
	Let $\alpha = (3, 1, \ldots, 2, 1, \ldots, 1)$ be a composition of length $n-1$ with $\beta =2$ in the $(k+1)$ position for some $k \in [1, n-2]$. Let $S_{3, 2}$ be the set of all such $\alpha$. Then
	\begin{align*}
		\sum_{\alpha \in S_{3, 2}} g_\alpha(p) &= \frac{1}{6(p-1)}\bigg(6(n-1)p^{2n-5} + (n-1)(n-2)(n-3)p^{n}\\
		&  - 3(n-1)(n-2)(n-4)p^{n-1} +(n-1)(n-2)(5n-24) p^{n-2} \\
		&- 3(n-1)(n-3)(n-4)p^{n-3}\bigg).
	\end{align*}
\end{cor}

\begin{lemma} \label{lemma222}
	Let $\alpha = (2, 1, \ldots, 2, 1, \ldots,2, 1, \ldots 1)$ be a composition of length $n-1$ with a 2 in the 1, $(k+1)$, and $(\ell+1)$ positions for $1 \leq k < \ell \leq n-2 $. Then 
	\begin{align*}
		g_\alpha(p) &=  p^{3n -k-\ell-9} + (n-\ell-1) (p-1) (p+1)p^{2n -k-7} \\
		&+(n-k-1)(p-1) p^{2n-\ell-6} 	 - (n-\ell-1)p^{n-4}(p-1)  \\
		&+(n-k-2)(n-\ell-1)  (p-1)^2p^{n-4}\\
		& + \left((n-\ell-2) + \binom{n-\ell-2}{2}\right)(p-1)(p-2)(p-3)p^{n-4}.
	\end{align*}
\end{lemma}

\vspace{.05in}
\begin{cor} \label{cor222}
	Let $\alpha = (2, 1, \ldots, 2, 1, \ldots,2, 1, \ldots 1)$ be a composition of length $n-1$ with a 2 in the 1, $(k+1)$, and $(\ell+1)$ positions for $1 \leq k < \ell \leq n-2 $. Let $S_{2,2,2}$ be the set of all such compositions. Then
	\begin{align*}
		\sum_{\alpha \in S_{2,2,2}} g_\alpha(p) &= \frac{p^{n-4}}{24 (p-1)^2 (p+1)} \bigg(12 ( n-2)(n-3) p^{ n+1}+ 24 p^{n}\\
		& - 24 (n-1)(n-3) p^{n-1} - 24 p^{n-2} + 12  n(n-3) p^{n-3}\\
		&+ 24 p^{2 n-5} + (n-1)(n-2)(n-3)(n-4) p^{6}\\
		&- 4(n-1)(n-2)(n-3)(n-5) p^{5} + (n-1)(n-2)(n-3)(7n-44) p^4 \\
		&- 	4 (n-1)(n-2)(n^2-11n+27) p^3- (n-3) (5 n^3-43n^2+82n-56) p^2\\
		&+ 	4 (n-3) (2 n^3-19n^2 +46n -26) p  -(n-3) (n-4) (n-5) ( 3 n-2) \bigg).
	\end{align*}
\end{cor}
We can now state the proof of Theorem \ref{n+2indexthm}.
\begin{proof}[Proof of Theorem \ref{n+2indexthm}]
Observe that $$g_n(p^{n+2}) = g_{n-1}(p^{n+1}) + \sum_{\alpha \in C_{n,n+2}} g_\alpha(p)$$ where $C_{n,n+2}$ is the set of all compositions of $n+2$ into $n-1$ parts. We proceed by induction on $n$. We use Proposition \ref{gn+1},  Lemma \ref{lemma4}, Corollaries  \ref{cor2beta},\ref{cor32} and \ref{cor222}, and the induction hypothesis. Then we simplify the expression.
\end{proof}

\subsection{Increasing difficulty in proving that $g_\alpha(p)$  is polynomial}
In order to prove that $g_{n}(p^{n+2})$ is a polynomial in $p$, we needed to consider all $\alpha$ which are compositions of $n+2$ into $n-1$ parts. In general, we can understand $f_n(p^e)$ recursively by considering $g_\alpha(p)$ for all $\alpha \in C_{n,e}$ and using Proposition \ref{recurrence}. For any fixed diagonal $\alpha = (e_1, \ldots, e_{n-1})$, we determine $g_\alpha(p)$ by counting the number of irreducible subring matrices of the form
$$
A = \begin{pmatrix}
	p^{e_1} & pa_{12} & pa_{13} & \cdots & pa_{1(n-1)} &1\\
	& p^{e_2} & pa_{23} & \cdots &pa_{2(n-1)}&1 \\
	&& p^{e_3} & \cdots & pa_{3(n-1)}&1 \\
	&&& \ddots & & \vdots\\
	&&&&p^{e_{n-1}}&1 \\
	&&&&&1
\end{pmatrix}.
$$

As discussed in Section \ref{background}, by our previous work in \cite{ish_subrings}, counting such irreducible subring matrices essentially reduces to counting the simultaneous solutions to a set of polynomials modulo powers of $p$. We conclude this paper by discussing the difficulty in counting the number of such systems of polynomials as $n$ and $e$ increase.

\begin{example}\label{first_diag_ex}
	Consider the diagonal $\alpha = (2,3,2,2).$ By employing the row reduction method, we compute $g_\alpha(p)$ by counting irreducible subring matrices of the form
	$$
	\begin{pmatrix}
		p^2 & pa_1 & pa_2 & pa_3&1\\
		&p^3 & pa_4 & pa_5&1\\
		&&p^2 &pa_6&1\\
		&&&p^2&1\\
		&&&&1
	\end{pmatrix}
	$$
	where $a_i \in [0,p)$ for $i \in \{1, 2, 3, 6\}$ and $a_4, a_5 \in [0,p^2)$. Notice that the condition $v_2 \circ v_2 \in \text{Col}(A)$ implies that $p \mid a_4$ so we can replace the entry $pa_4$ with $p^2a_4.$ By employing the row reduction method to the matrix
		$$
	\begin{pmatrix}
		p^2 & pa_1 & pa_2 & pa_3&1\\
		&p^3 & p^2a_4 & pa_5&1\\
		&&p^2 &pa_6&1\\
		&&&p^2&1\\
		&&&&1
	\end{pmatrix}
	$$
	the closure conditions simplify to
	\begin{align*}
	&	\frac{-a_1a_4a_5 + a_1a_4a_6}{p}\in \Z\\
	&	\frac{a_1a_4a_6^2 - a_1a_4a_6p - a_2a_6^2p  - a_1a_5^2 + a_1a_5}{p^2}\in \Z\\ 
	&\frac{-a_4a_6^2 + a_5^2}{p} \in \Z.
	\end{align*}

Observe that the numerators essentially define a system of polynomials that must vanish over $\Z/p^2\Z$ though the variables $a_i$ are restricted between $[0,p)$ for $i\ne 6.$ The first condition implies that $p \mid a_1, p \mid a_5,$ or $p \mid (a_5-a_6)$. If $p \mid a_1$, then $a_1 \in [0,p)$ implies that $a_1 =0$. Thus we can simplify the conditions to 
	\begin{align*}
	&	\frac{ - a_2a_6^2}{p}\in \text{Col}(A)\\ 
	&\frac{-a_4a_6^2 + a_5^2}{p} \in \text{Col}(A)
\end{align*}
The number of solutions to these conditions is equal to $p \cdot \#V(- a_2a_6^2, -a_4a_6^2 + a_5^2) $ where $V \subset \mb{A}^6(\F_p).$ Further case reduction gives that the number of solutions to the closure conditions under the assumption that $a_1 =0$ is $p \cdot (2p^3 -p^2) = p^3(2p-1).$

In order to give an exact formula for $g_\alpha(p)$, we would proceed by implementing similar arguments for the cases $p \mid a_5$ and $p \mid (a_5-a_6)$.
\end{example}

If one is only interested in whether or not $g_\alpha(p)$ is polynomial, the work outlined in Example \ref{first_diag_ex} is a bit easier. For example, we can tell that $V(- a_2a_6^2, -a_4a_6^2 + a_5^2) $ has a polynomial number of $\F_p$-rational points by just studying the equations rather than computing a polynomial formula.

\begin{example}
	Consider the diagonal $\alpha =(3,2,2,1,1,1)$. Determining whether $g_\alpha(p)$ is polynomial reduces to understanding solutions to 16 equations in 12 variables with some equations equal to 0 modulo $p$, and others equal to 0 modulo $p^2$. It is currently unknown whether $g_\alpha(p)$ is polynomial in this case.
\end{example}

As $n$ and $e$ increase, it seems infeasible to compute $g_\alpha(p)$ even with the aid of a computer algebra system like Sage \cite{sage} or Magma \cite{magma}. The work is especially difficult since we do not understand large systems of polynomials modulo several powers of $p$ very well. The strategy outlined above works best when we can show that the system of equations can be defined over $\Z/p\Z$. In this case, we can identify the system of polynomials with a variety defined over $\F_p$ and use techniques from arithmetic geometry to understand the corresponding variety.

While the row reduction strategy simplifies the work needed to understand $g_\alpha(p)$, it is not sufficient in the goal of proving $f_n(p^e)$ is polynomial for fixed $n$ and $e$ or searching for counterexamples to Question \ref{poly_ques}. We hope that future researchers will be motivated by this work to investigate Question \ref{poly_ques} further -- it seems likely that a new approach is needed.

\section{Acknowledgments}
This work was supported by the NSF grant DMS 1802281. The author thanks Nathan Kaplan for many helpful conversations that greatly improved the presentation of this paper. Parts of this work appeared in the author's Ph.D. thesis.

\appendix
\section{Proofs of Several Lemmas and Corollaries}
First, we prove Lemma \ref{lemma3beta} and Corollary \ref{cor32}.
\begin{proof}[Proof of Lemma \ref{lemma3beta}]
	
	Let $\alpha = (3, 1, \ldots, \beta, 1, \ldots, 1)$ be a composition of length $n-1$ with $\beta >1$ in the $(k+1)$ position for some $k \in [1, n-2]$. Consider the matrix
	$$
	\begin{pmatrix}
		p^3 & pa_1 & pa_2 & \cdots  & pa_k & pa_{k+1}&\cdots & pa_{n-2}&1\\
		& p & 0 &0  & 0&0&\cdots &0&1\\
		&  & p &0 &0 &0 &\cdots &0&1\\
		&& &\ddots  &\vdots&\vdots&&\vdots&\vdots\\
		&&&&p^\beta & pb_{k+1}& \cdots & pb_{n-2}&1\\
		&&&&&p&&0&1\\
		&&&&&&\ddots&\vdots&\vdots\\
		&&&&&&&p&1
	\end{pmatrix}.
	$$
	By augmenting the vector $v_i \circ v_j$ and row reducing, we obtain the following expressions that must simultaneously be integers
	\begin{align*}
		& \frac{a_i^2 - a_i}{p} \in \Z \text{ for } 1 \leq i < k\\
		& \frac{a_i a_j}{p} \in \Z  \text{ for } 1 \leq i  < k \text{ and for any }  j \ne i\\
		& \frac{a_k^2-a_kp^{\beta-1}}{p} \in \Z \\
		&\frac{a_k(a_j - b_j)}{p} \in \Z \text{ for } j > k\\
		&\frac{b_i b_j}{p^{\beta-2}} \in \Z \text{ for any } k < i < j\\
		& \frac{a_i a_j}{p} - \frac{a_k b_ib_j}{p^{\beta}}\in \Z \text{ for any } k < i < j\\
		&\frac{b_j^2 - b_j}{p^{\beta-2}} \text{ for any } j > k\\
		& \frac{a_j^2 - a_j}{p} - \frac{(b_j^2 -b_j)a_k}{p^\beta} \in \Z\text{ for any } j > k.
	\end{align*}
	
	The first equation comes from $v_i \circ v_i$ for $i  <k$. The second equation comes from $v_i \circ v_j$ where $i < k$ and $j \ne i$. The third equation is $v_k \circ v_k$. The fourth equation is from $v_k \circ v_j$ for $j > k$. The fifth and sixth equations are $v_i \circ v_j$ where $i\ne j$ and $i,j > k$. The last two equations are from $v_j \circ v_j $ where $j > k$. 
	
	By the first equation, $p \mid (a_i-1)$ or $p \mid a_i$ for all $i < k$. Further, $p \mid (a_i-1)$ for at most one $a_i$.
	\vspace{.1in}
	
	\noindent \textbf{Case 1:}  If $p \mid (a_i-1)$ for some $i$, then $p\mid a_j$ for all $i \ne j$. This reduces the equations to
	\begin{align*}
		& \frac{a_k}{p} \in \Z\\
		&\frac{b_i b_j}{p^{\beta-2}} \in \Z \text{ for any } k < i < j\\
		&  \frac{a_k b_ib_j}{p^{\beta}}\in \Z \text{ for any } k < i < j\\
		&\frac{b_j^2 - b_j}{p^{\beta-2}} \text{ for any } j > k\\
		&  \frac{(b_j^2 -b_j)a_k}{p^\beta} \in \Z\text{ for any } j > k.
	\end{align*}

	\noindent \textbf{Case 2:} If $p \mid a_i$ for all $i < k$, the equations reduce to 
	\begin{align*}
		& \frac{a_k}{p} \in \Z \\
		&\frac{b_i b_j}{p^{\beta-2}} \in \Z \text{ for any } k < i < j\\
		&\frac{a_i a_j}{p} - \frac{a_k b_ib_j}{p^{\beta}}\in \Z \text{ for any } k < i < j\\
		&\frac{b_j^2 - b_j}{p^{\beta-2}} \text{ for any } j > k\\
		&  \frac{a_j^2 - a_j}{p} -\frac{(b_j^2 -b_j)a_k}{p^\beta} \in \Z\text{ for any } j > k.
	\end{align*}
	\vspace{.1in}
	In either case, $a_k =dp$ for some $d \in [0,p)$. The number of solutions will be different depending on whether $\beta = 2$ or $\beta > 2$. 
	
	\vspace{.1in}
	\noindent\textbf{\Large Proof when $\beta = 2$:}
	\begin{table}[!htbp]\centering
		\begin{tabular}{|c|c|c|}
			\hline
			Case &	$d$ & Number of Solutions\\[6pt]
			\hline
			1A&	$d=0$ & $(k-1)p^{2n-k-5}$\\[6pt]
			\hline
			1B&	$d \ne 0$ & $(n-k-1)(k-1)p^{n-3}(p-1)$\\[6pt]
			\hline
		\end{tabular}
		\caption{Cases when $p \mid (a_i-1)$ for some $i < k$}
	\end{table}

	\noindent \textbf{Proof of Case 1:}
	Suppose that $p \mid (a_i -1)$ for some $i < k$ and $p \mid a_j$ for all $j \ne i$. Set $a_k = dp$ where $d \in [0,p)$. Given these conditions, $v_i \circ v_j \in \text{Col}(A)$ if and only if the following conditions are satisfied:
	\begin{align*}
		&  \frac{db_ib_j}{p}\in \Z \text{ for any } k < i < j\\
		&  \frac{(b_j^2 -b_j)d}{p} \in \Z\text{ for any } j > k.
	\end{align*}
	
	Observe that $a_i \in [0,p^2)$ for all $1\leq i \leq n-2$. Since all equations depend on $a_i \pmod{p}$, we solve for the congruence class of $a_i \pmod{p}$, $i \ne k$, and then multiply the final equations by $p^{n-3}$. 
	\vspace{.1in}
	
	\noindent\textbf{Proof of Case 1A:}
	If $d =0$, the two conditions are trivially satisfied. Since $b_i \in [0,p)$, there are $p$ choices for each $b_i$. Further, there are $p$ choices for each $a_i$, $i \ne k$ and there are $k-1$ ways to choose which $a_i \equiv 1\pmod{p}$. This gives
	$$
	(k-1) p^{n-3} \cdot p^{n-k-2} = (k-1)p^{2n-k-5}
	$$
	solutions.
	\vspace{.1in}
	
	\noindent\textbf{Proof of Case 1B:}
	If $d\ne 0$, the conditions simplify to:
	\begin{align*}
		&  \frac{b_ib_j}{p}\in \Z \text{ for any } k < i < j\\
		&  \frac{(b_j^2 -b_j)}{p} \in \Z\text{ for any } j > k.
	\end{align*}
	Either all $b_i = 0$ or there is exactly one $b_i = 1$. In either case, there is exactly one choice for each $b_i$ and there are $p$ choices for each $a_i$, $i \ne k$. There are $k-1$ ways to choose which $a_i \equiv 1\pmod{p}$. This gives
	$$
	(n-k-1)(k-1)p^{n-3}(p-1)
	$$
	solutions.
	
	There are a total of 
	$$
	(k-1)p^{2n-k-5}+ (n-k-1)(k-1)p^{n-3}(p-1)
	$$
	solutions in Case 1.
	
	\vspace{.1in}
	\begin{table}[!htbp]\centering
		\begin{tabular}{|c|c|p{3cm}|p{3cm}|c|}
			\hline
			Case & $d$ & $b_i$  & \text{Other Conditions} &  Number of Solutions\\[5pt]
			\hline
			2A&	$d=0$ & && $(n-k-1) p^{2n - k -5}$\\[5pt]
			\hline
			2B&	$d \ne 0$ & $b_i \in \{0,1\}$ \text{ for all } $i > k$&& $(n-k-1)^2 p^{n-3}(p-1)$\\[5pt]
			\hline
			2C(i)&	$d =1$ & $b_i \not \in \{0,1\}$ \text{ for some } $i > k$ & $a_j \equiv b_j \pmod{p}$ for all $j > k$ &$ p^{2n-k-5}  - (n-k-1)p^{n-3}$\\[5pt]
			\hline
			2C(ii)&	$d \ne 0$ & $b_i \not \in \{0,1\}$ \text{ for some } $i > k$ & $a_j \equiv 0 \pmod{p}$ for all $j > k$ & $(n-k-2) p^{n-3}(p-2)(p-3)$\\[5pt]
			\hline
			2C(iii)&	$d \ne 0$ & $b_i \not \in \{0,1\}$ \text{ for some } $i > k$ & $a_j \equiv 1-a_i \pmod{p}$ for exactly one $j \ne i,k$ & $\binom{n-k-2}{2}p^{n-3} (p-2)(p-3)$\\[5pt]
			\hline
		\end{tabular}
		\caption{Cases when $p \mid a_i$ for all $i < k$}
	\end{table}
	\vspace{.1in}
	
	\noindent \textbf{Proof of Case 2:}
	Suppose that $p \mid a_i$ for all $i  <k$. Then $v_i \circ v_j \in \text{Col}(A)$ for all $1 \leq i \leq j \leq n-2$ if and only if 
	\begin{align}
		&\frac{a_i a_j}{p} - \frac{db_ib_j}{p}\in \Z \text{ for any } k < i < j \label{eqnC2.1}\\
		&  \frac{a_j^2 - a_j}{p} -\frac{(b_j^2 -b_j)d}{p} \in \Z\text{ for any } j > k\label{eqnC2.2}.
	\end{align}
	As in Case 1, we solve for $a_i \pmod{p}$ and multiply the number of solutions by $p^{n-3}$. 
	\vspace{.1in}
	
	\noindent\textbf{Proof of Case 2A:}
	If $d =0$, these conditions reduce to
	\begin{align*}
		&\frac{a_i a_j}{p}\in \Z \text{ for any } k < i < j \\
		&  \frac{a_j^2 - a_j}{p} \in \Z\text{ for any } j > k.
	\end{align*}
	The $b_i \in [0,p)$ are arbitrary and we solve these equations for the $a_i$ as before. We have a total of
	$$
	(n-k-1) p^{n-3} \cdot p^{n-k-2} = (n-k-1) p^{2n - k -5}
	$$
	solutions.
	\vspace{.1in}
	
	Now suppose $d \ne 0$. If any $a_i \equiv 0, 1\pmod{p}$ for $k < i < n-1$, then $p \mid (b_i^2 - b_i)$. Since $b_i \in [0,p)$, then $b_i =0$ or 1. Similarly, if $b_i =0,1$ for some $k < i < n-1$, then $a_i \equiv 0, 1 \pmod{p}$. 
	\vspace{.1in}
	
	\noindent \textbf{Proof of Case 2B:}
	Suppose all $b_i \in\{0,1\}$ for $i > k$. Then all $a_i \equiv 0,1 \pmod{p}$ for $i > k$.
	
	Observe that if $b_i = b_j = 1$ for $i \ne j$, then $a_i a_j \equiv d \pmod{p}$ and since $d \ne 0$, then $a_i, a_j \equiv 1 \pmod{p}$. Thus $d= 1$ and $a_j \equiv b_j \pmod{p}$ for all $j > k$. We include this count later in Case 2C(i). We are left to count the number of solutions to Case 2B with at most one $b_i =1$.
	
	If at most one $b_i =1$, then $a_i a_j \equiv 0 \pmod{p}$ for all $i \ne j > k$ and so at most one $a_j\equiv 1 \pmod{p}$. This gives
	$$
	(n-k-1)^2 p^{n-3}(p-1)
	$$
	solutions, choosing the indices $i, j$ so that $b_i =1$ and $a_j \equiv 1 \pmod{p}$. 
	\vspace{.1in}
	
	\noindent \textbf{Proof of Case 2C:}
	Suppose that not all $b_i \in \{0,1\}$. Pick any $b_i \ne 0,1 $ (so $a_i \not \equiv 0,1 \pmod{p})$. We can solve for $d$ using \eqref{eqnC2.2}. This gives
	\begin{equation} \label{eqnd}
		d  = \frac{a_i^2 - a_i}{b_i^2 - b_i}.
	\end{equation}
	We plug this into \eqref{eqnC2.1}, giving
	$$
	a_i a_j = \frac{a_i^2 - a_i}{b_i^2 - b_i} b_i b_j.
	$$
	Rearranging, we can solve for $b_j$ since $a_i \not \in \{0,1\}$. Thus
	\begin{equation} \label{eqnbj}
		b_j  = \frac{(b_i-1)a_j}{a_i -1}.
	\end{equation}
	Therefore all $b_j$ are determined by $a_i, b_i$, and $a_j$ for $j \ne i$. Now, plugging \eqref{eqnd} and \eqref{eqnbj} into \eqref{eqnC2.2}, we have
	$$
	a_j^2 - a_j -  \frac{a_i^2 - a_i}{b_i^2 - b_i}\left( \left( \frac{(b_i-1)a_j}{a_i -1}\right)^2 -  \frac{(b_i-1)a_j}{a_i -1}\right) \equiv 0 \pmod{p}.
	$$
	We simplify to obtain
	$$
	a_j^2 - a_j - \frac{-a_ia_j^2(b_i-1) + (a_i-1)a_i a_j}{(a_i - 1)b_i} \equiv 0 \pmod{p}.
	$$
	We then multiply by the denominator and factor to find
	\begin{equation}\label{eqnaj}
		a_j (1-a_j -a_i)(b_i -a_i) \equiv 0 \pmod{p}.
	\end{equation}
	Either $a_j \equiv 0 \pmod{p}$, $a_j \equiv 1 - a_i \pmod{p}$, or $a_i \equiv b_i \pmod{p}$ for each $j  > k$. As we have seen, if $a_i \equiv b_i \pmod{p}$, then $d=1$ and $a_\ell \equiv b_\ell \pmod{p}$ for all $\ell > k$. We count this case later in Case 2C(i). Otherwise, we see that $a_j$ is determined and we are left to understand the equation $a_i^2 - a_i - d(b_i^2 - b_i) \equiv 0 \pmod{p}$. 
	
	Suppose that $1- a_j - a_i\equiv 0 \pmod{p}$ and $1- a_\ell - a_i \equiv 0 \pmod{p}$ for two indices $j , \ell > k$ and $j \ne \ell$. Then $a_j \equiv a_\ell \equiv 1-a_i \pmod{p}$. We plug this into \eqref{eqnbj} for indices $j, \ell$ and we find that $b_j = b_\ell=1-b_i$. We can solve $d \equiv \frac{a_j a_\ell}{b_j b_\ell} \pmod{p}$ using \eqref{eqnC2.1}. Setting the two expressions for $d$ equal to each other and substituting, we find
	$$
	(a_i-1)b_i \equiv (b_i-1)a_i \pmod{p}.
	$$
	This implies  $b_i \equiv a_i \pmod{p}$. We substitute this and the formula for $a_j$ into the first type of equation to obtain
	$$
	a_i (1-a_i) - da_i(1-a_i ) \equiv 0 \pmod{p}.
	$$
	But this equation implies $a_i \equiv 1\pmod{p}$. Thus we have reached a contradiction as we assumed $b_i, a_i \not \equiv 0,1 \pmod{p}$. So we can have at most one $a_j \equiv 1-a_i \pmod{p}$.
	\vspace{.1in}
	
	\noindent\textbf{Proof of Case 2C(i):}
	Suppose that $a_i \equiv b_i \pmod{p}$. Then $d = 1$ and $a_j \equiv b_j \pmod{p}$ for all $j > k$. This gives
	$$
	p^{n-3} \cdot p^{n-k-2}  - (n-k-1)p^{n-3}
	$$
	solutions discounting the case when all $b_i$ and $a_i$ are in $\{0,1\}$ (which just depends on picking the index $i$ so that $b_i \equiv a_i \equiv 1 \pmod{p}$ or picking all $b_i \equiv  a_i \equiv  0 \pmod{p}$).
	
	This leaves two cases: all $a_j \equiv 0\pmod{p}$ for $i \ne j$ and exactly one $a_j \equiv 1-a_i \pmod{p}$. For both cases, observe that all $b_j , a_j$ for $j \ne i$ are determined and we are left to solve 
	$$
	a_i^2 - a_i - d(b_i^2 - b_i) \equiv 0 \pmod{p}.
	$$
	
	\noindent\textbf{Proof of Case 2C(ii):}
	If all $a_j \equiv 0 \pmod{p}$ for $j \ne i, k$, then all other variables are fixed (and all $b_j =0$ for $j\ne i,k$). Since $b_i \ne 0,1$, then we can solve for $d$ using \eqref{eqnd}. We multiply by $(n-k-2)$ to choose the $i$ so that $a_i, b_i \not \equiv 0,1 \pmod{p}$. 
	$$
	(n-k-2) p^{n-3}(p-2)(p-3)
	$$
	solutions. Note that $a_i \not \equiv b_i \pmod{p}$ since we counted this in Case 2C(i).
	\vspace{.1in}
	
	\noindent\textbf{Proof of Case 2C(iii):}
	Suppose exactly one $a_j \equiv 1- a_i \pmod{p}$ and all other $a_\ell \equiv 0\pmod{p}$. Note this implies $b_j =1-b_i$ and $b_\ell = 0$ for all other $\ell >k$. We find the number of solutions to this equation and then multiply by $\binom{n-k-2}{2}$ to pick the two indices $i,j$ choosing the $a_i$ and $a_j \equiv 1-a_i \pmod{p}$. Since $b_i \ne 0,1$, then we can solve for $d$ using \eqref{eqnd} as before and this leaves $a_i, b_i \not \equiv 0,1 \pmod{p}$. Again, we want $a_i \not \equiv b_i \pmod{p}$ since this case was counted earlier. Thus there are
	$$
	\binom{n-k-2}{2}p^{n-3} (p-2)(p-3)
	$$
	solutions in this case.
	\vspace{.1in}
	
	\noindent \textbf{\Large Proof when $\beta > 2$}\\
	\begin{table}[!htbp]\centering
		\begin{tabular}{|c|c|c|}
			\hline
			Case &	$d$ & Number of Solutions\\[6pt]
			\hline
			1A&	$d=0$    & $(n-k-1)(k-1) p^{2n-k-5}$\\[6pt]
			\hline
			1B&	$d \ne 0$ & $(n-k-1)(k-1)p^{n-3}(p-1)$\\[6pt]
			\hline
		\end{tabular}
		\caption{Cases when $p \mid (a_i-1)$ for some $i < k$}
	\end{table}
	
	\vspace{.1in}
	
	\noindent \textbf{Proof of Case 1:} Suppose $p \mid (a_i -1)$ for some $i < k$, so $p \mid a_j$ for all $j \ne i$. Let $a_k = dp$ for some $d \in [0, p)$. Then $v_i \circ v_j \in \text{Col}(A)$ for all $ 1\leq i < j \leq n-2$ if and only if the following equations are satisfied
	\begin{align}
		&\frac{b_i b_j}{p^{\beta-2}} \in \Z \text{ for any } k < i < j\label{eqngC1.1}\\
		&  \frac{d b_ib_j}{p^{\beta-1}}\in \Z \text{ for any } k < i < j\label{eqngC1.2}\\
		&\frac{b_j^2 - b_j}{p^{\beta-2}} \text{ for any } j > k\label{eqngC1.3}\\
		&  \frac{(b_j^2 -b_j)d}{p^{\beta-1}} \in \Z\text{ for any } j > k\label{eqngC1.4}.
	\end{align}
	
	Since the $a_j$ do not appear in the above equations, we can choose any $a_j \in [0,p^2)$ so that $p \mid a_j$ for $j \ne i, k$ and $p \mid (a_i-1)$. We multiply each subcase in Case 1 by $p^{n-3}$.
	\vspace{.1in}
	
	\noindent \textbf{Proof of Case 1A:} If $d =0$, then conditions \eqref{eqngC1.2} and \eqref{eqngC1.4} vanish. Thus we need to solve $p^{\beta-2} \mid b_i b_j$ for all $k < i < j$ and $p^{\beta-2} \mid b_j(b_j-1)$ for all $j > k$. Recall that $b_i \in [0,p^{\beta-1})$ for all $i > k$. Observe that $p^{\beta-2} \mid (b_j-1)$ for at most one $j > k$. In either case, there are $p$ choices for $b_i, i > k$ and there are $n-k-1$ ways to pick the index $j$ so that $p^{\beta-2} \mid (b_j-1)$ or pick no such index. There are $k-1$ ways to choose the $a_i \equiv 1 \pmod{p}$ for some $i < k$. This leads to 
	$$
	(n-k-1) (k-1)p^{n-3} p^{n-k-2} = (n-k-1)(k-1) p^{2n-k-5}
	$$
	solutions.
	\vspace{.1in}
	
	\noindent \textbf{Proof of Case 1B:}
	Suppose $d \ne 0.$ In this case, conditions  \eqref{eqngC1.2} and \eqref{eqngC1.4} do not vanish. These are equivalent to $p^{\beta-1} \mid b_j(b_j-1)$ for $j > k$ and $p^{\beta-1} \mid b_i b_j$ for $k > j > i$. In particular, \eqref{eqngC1.4} implies that $p^{\beta-1} \mid b_j$ or $p^{\beta-1} \mid (b_j-1)$ for all $j >k$. Since $b_j \in [0, p^{\beta-1})$, then $p^{\beta-1} \mid b_j$ implies $b_j =0$ and $p^{\beta-1} \mid (b_j-1)$ implies that $b_j =1$.
	
	Observe that there exists at most one index $j >k$ so that $b_j=1$. There are $n-k-1$ ways to choose at most one index $j > k$ so that $b_j=1$. There are $k-1$ ways to choose the index $i$ so that $a_i \equiv 1 \pmod{p}$. This gives
	$$
	(k-1)(n-k-1) p^{n-3} (p-1) 
	$$
	solutions.

	\noindent \textbf{Proof of Case 2:} Suppose $p \mid a_i$ for all $i < k$ and $a_k = dp$ for some $d \in [0,p)$. The conditions reduce to 
	\begin{align}
		&\frac{b_i b_j}{p^{\beta-2}} \in \Z \text{ for any } k < i < j \label{eqngC2.1}\\
		&\frac{a_i a_j}{p} - \frac{d b_ib_j}{p^{\beta-1}}\in \Z \text{ for any } k < i < j\label{eqngC2.2}\\
		&\frac{b_j^2 - b_j}{p^{\beta-2}} \text{ for any } j > k\label{eqngC2.3}\\
		&  \frac{a_j^2 - a_j}{p} -\frac{(b_j^2 -b_j)d}{p^{\beta-1}} \in \Z\text{ for any } j > k\label{eqngC2.4}.
	\end{align}
		\begin{table}[!htbp]\centering
		\begin{tabular}{|c|c|p{2cm}|p{2cm}|c|}
			\hline
			Case & $d$ & $b_i$  & Other Conditions &  Number of Solutions\\[5pt]
			\hline
			2A&	$d=0$ & && $(n-k-1)^2 p^{2n-k-5}$\\[5pt]
			\hline
			2B(i)a&	$d \ne 0$ &$p^{\beta-2} \mid b_j$ for all $j > k$ & $p \mid a_i$ for all $i > k$& $p^{n-3}(p-1)$\\[5pt]
			\hline
			2B(i)b&	$d \ne 0$ &$p^{\beta-2} \mid b_j$ for all $j > k$ & $p \mid (a_i-1)$ for one $i > k$& $(n-k-2)p^{n-3}(p-1)^2$\\[5pt]
			\hline
			2B(ii)a & $d \ne 0$ & $p^{\beta-2} \mid (b_j-1)$ for one $j > k$ & $p \mid a_i$ for all $i  > k, i\ne j$ &  $(n-k-2)p^{n-2} (p-1)$\\[5pt]
			\hline
			2B(ii)b1 & $d \ne 0$ & $p^{\beta-2} \mid (b_j-1)$ for one $j > k$ & $p \mid(a_i-1)$ for one $i  > k, i\ne j$ &  $(n-k-2)(n-k-3) p^{n-3}(p-1)$\\[5pt]
			\hline
			2B(ii)b2 & $d \ne 0$ & $p^{\beta-2} \mid (b_j-1)$ for one $j > k$ & $a_i \not \equiv 0,1 \pmod{p}$ for one $i  > k, i\ne j$ &  $(n-k-2)(n-k-3)p^{n-3}(p-1)(p-2)$\\[5pt]
			\hline
		\end{tabular}
		\caption{Cases when $p \mid a_i$ for all $i < k$}
	\end{table}
	\vspace{.1in}
	
	\noindent \textbf{Proof of Case 2A:} If $d =0$, the conditions reduce further,
	\begin{align*}
		&\frac{b_i b_j}{p^{\beta-2}} \in \Z \text{ for any } k < i < j\\
		&\frac{b_j^2 - b_j}{p^{\beta-2}} \text{ for any } j > k\\
		&\frac{a_i a_j}{p}\in \Z \text{ for any } k < i < j\\
		&  \frac{a_j^2 - a_j}{p}  \in \Z\text{ for any } j > k.
	\end{align*}
	As before, either all $b_j$ are divisible by $p^{\beta-2}$ or exactly one $b_j-1$ is divisible by $p^{\beta-2}$. Similarly, all $a_i \equiv 0 \pmod{p}$ or exactly one $a_i \equiv 1 \pmod{p}$. This gives
	$$
	(n-k-1)^2 p^{n-3} p^{n-k-2}
	$$
	solutions.
	\vspace{.1in}
	
	\noindent\textbf{Proof of Case 2B:} Suppose $d \ne 0$. The conditions \eqref{eqngC2.1} and \eqref{eqngC2.3} tell us that either $p^{\beta-2} \mid b_j$ for all $j > k$ or exactly one $b_j -1$ is divisible by $p^{\beta-2}$.
	\vspace{.1in}
	
	\noindent\textbf{Proof of Case 2B(i):} Suppose $p^{\beta-2} \mid b_j$ for all $j > k$. Certainly $p^{\beta-1} \mid b_i b_j$ for $i > k, i \ne j$. Therefore condition \ref{eqngC2.2} reduces to $p \mid a_i a_j$ for all $k < i  < j$. Thus at most one $a_i$ is not divisible by $p$.
	\vspace{.1in}
	
	\noindent\textbf{Proof of Case 2B(i)a:}
	If all $a_i$ are divisible by $p$, then we require $p^{\beta-1} \mid (b_j^2-b_j)$ for all $j > k$. Thus $p^{\beta-1} \mid b_j$ for all $j > k$ and so $b_j =0$ for all $j > k$. We have
	$$
	p^{n-3}(p-1)
	$$
	solutions.
	\vspace{.1in}
	
	\noindent\textbf{Proof of Case 2B(i)b:}
	If one $a_j \not \equiv 0 \pmod{p}$, then \eqref{eqngC2.4} does not reduce for index $j$, but does for indices $i\ne j, k$. That is, \eqref{eqngC2.4} becomes $\frac{a_j^2 - a_j}{p} -\frac{(b_j^2 -b_j)d}{p^{\beta-1}} \in \Z$ and $\frac{(b_i^2-b_i)d}{p^{\beta-1}}\in\Z$ for all $i \ne j$. Since $p^{\beta-2} \mid b_i$ and $p^{\beta-1} \mid (b_i^2-b_i)$, then all $b_i =0$ for $i \ne j$. Since $p^{\beta-2} \mid b_j$, then $b_j = n_j p^{\beta-2}$ for some $n_j \in [0,p)$. Condition \eqref{eqngC2.4} for index $j$ reduces to
	$$
	p \mid (a_j^2 -a_j - n_j d).
	$$
	Since $d \ne 0$, we can solve for $n_j \equiv \frac{a_j^2 -a_j}{d} \pmod{p}$. Recall that $n_j \in [0,p), d \in (0,p),$ and $a_j \in (0,p)$. This gives
	$$
	(n-k-2)p^{n-3}(p-1)^2
	$$ 
	solutions. 
	\vspace{.1in}
	
	\noindent \textbf{Proof of Case 2B(ii):} Suppose $p^{\beta-2} \mid (b_j-1)$ for some $j > k$ and $p^{\beta-2} \mid b_i$ for all $i > k, i \ne j$. Write $b_i = p^{\beta-2} n_i$ where $n_i \in [0, p)$. We see that \eqref{eqngC2.4} reduces to 
	$\frac{(a_i^2 - a_i) + n_id}{p} \in \Z$
	for all $i > k, i \ne j$ and $ \frac{a_j^2 - a_j}{p} -\frac{(b_j^2 -b_j)d}{p^{\beta-1}} \in \Z$. Condition \eqref{eqngC2.2} reduces to $p \mid a_i a_{\ell}$ whenever $i, \ell> k$ are not equal to $j$ and $\frac{a_i a_j}{p} - \frac{d n_ib_j}{p} \in \Z$ whenever $i > k, i \ne j$. Thus we see that at most one $a_i \not \equiv 0 \pmod{p}$ with $i > k, i\ne j$.
	\vspace{.1in}
	
	\noindent \textbf{Proof of Case 2B(ii)a:} If all $a_i \equiv 0 \pmod{p}$ for $i > k, i\ne j$, then we are left with the conditions
	\begin{align*}
		& \frac{dn_i b_j}{p} \in \Z\\
		&  \frac{a_j^2 - a_j}{p} -\frac{(b_j^2 -b_j)d}{p^{\beta-1}}  \in \Z\text{ for any } j > k.
	\end{align*}
	Since $p \nmid d, b_j$, the first equation is satisfied if $p \mid n_i$ for all $i>k$, that is, if $p^{\beta-1} \mid b_i$ for all $i > k, i \ne j$. This implies that $b_i= 0$ for all $i\ne j, i > k$. Since $p^{\beta-2} \mid (b_j-1)$, set $b_j - 1 = r_j p^{\beta-2}$ for some $r_j \in [0,p)$. We are left to solve
	$$
	(a_j^2 - a_j - b_j d r_j) \equiv 0 \pmod{p}.
	$$
	Observe that $b_j = 1 + r_j p^{\beta-2}$, $d \in (0,p)$, and $r_j \in [0,p)$. We can plug in to obtain
	$$
	(a_j^2 - a_j - dr_j) \equiv 0 \pmod{p}.
	$$
	Since $p \nmid d$, we can solve $r_j \equiv \frac{a_j^2 - a_j}{d} \pmod{p}$. Recall that $r_j \in [0,p)$. We have $a_j \in [0,p)$ and $d \in (0,p)$. Further all $b_i$ are fixed. Here we have
	$$
	(n-k-2)p^{n-3} (p-1)  \cdot p = (n-k-2)p^{n-2} (p-1)
	$$
	solutions.
	\vspace{.1in}
	
	\noindent \textbf{Proof of Case 2B(ii)b:} Suppose one $a_i \not \equiv 0 \pmod{p}$ for some $i > k, i \ne j$. We are left with the conditions
	\begin{align}
		& \frac{(a_i^2 - a_i) + n_id}{p} \in \Z \label{eqng2bii1}\\
		& \frac{a_ia_j}{p} - \frac{dn_i b_j}{p} \in \Z\label{eqng2bii2}\\
		&  \frac{a_j^2 - a_j}{p} -\frac{(b_j^2 -b_j)d}{p^{\beta-1}}  \in \Z\label{eqng2bii3}.
	\end{align}
	We can solve \eqref{eqng2bii2} for $n_i$ since $p \nmid d, b_j$. We obtain
	$$
	n_i \equiv \frac{a_i a_j}{db_j} \pmod{p}.
	$$
	We plug this into \eqref{eqng2bii1}. This gives
	$$
	(b_j(a_i^2 - a_i) + a_ia_jd) \equiv 0 \pmod{p}.
	$$
	By factoring, 
	$$
	a_i (b_j (a_i-1) + a_jd) \equiv 0 \pmod{p}.
	$$
	\vspace{.1in}
	
	\noindent \textbf{Proof of Case 2B(ii)b1:}
	If $a_i \equiv 1\pmod{p}$, then we see that $p \mid a_j$ and so $p^{\beta-1} \mid (b_j-1)$. Further, $p \mid dn_i b_j$ and so $p^{\beta-1} \mid b_i$ for all $i > k$. This gives
	$$
	(n-k-3)(n-k-2) p^{n-3}(p-1)
	$$
	solutions.
	\vspace{.1in}
	
	\noindent \textbf{Proof of Case 2B(ii)b2:} Otherwise, if $a_i \not \equiv 0,1 \pmod{p}$, then we can solve 
	\begin{equation} \label{eqnbj2}
		b_j \equiv \frac{d a_j}{1-a_i} \pmod{p}.
	\end{equation}
	Recall that $b_j  \equiv 1 \pmod{p}$ and so 
	$$
	d a_j \equiv 1 - a_i \pmod{p}.
	$$
	Thus we can solve for $a_j \equiv \frac{1-a_i}{d} \pmod{p}$. We plug this into \eqref{eqnbj2} to find $b_j \equiv a_j \pmod{p}$. Therefore \eqref{eqng2bii3} becomes
	$$
	\frac{a_j^2 - a_j}{p} - \frac{d(a_j^2 - a_j)}{p^{\beta-1}}.
	$$
	Since $b_j \equiv 1 \pmod{p^{\beta-2}},$ then so is $a_j$, so setting $a_j -1 = r_i p^{\beta-2}$ for some $r_i \in [0,p)$ gives
	$$
	p \mid d a_j r_i 
	$$
	and so we must have $a_j \equiv b_j \equiv 1 \pmod{p^{\beta-1}}$. We can choose any $a_i \in [1,p)$ and any $d \in (0,p)$. There is a fixed choice for all $b_\ell$ and for $a_j$. Thus we have
	$$
	(n-k-2)(n-k-3)p^{n-3}(p-1)(p-2)
	$$
	solutions.
\end{proof}

\begin{proof}[Proof of Corollary \ref{cor32}]
	For each term in the summand of $g_\alpha$ from Lemma \ref{lemma3beta}, we sum from $k=1$ to $n-2$. We then simplify each term. Each individual sum is found using Mathematica \cite{mathematica}.
\end{proof}

Lastly, we give the proofs of Lemma \ref{lemma222} and Corollary \ref{cor222}.

\begin{proof}[Proof of Lemma \ref{lemma222}]
	Consider the matrix
	\setcounter{MaxMatrixCols}{15}
	$$
	\begin{pmatrix}
		p^2 & pa_1 &  \cdots  & pa_k & pa_{k+1}& \cdots &pa_{\ell} & pa_{\ell+1} &\cdots & pa_{n-1}&1\\
		& p  &\cdots &0&0&\cdots &0 &0&\cdots&0&1\\
		&&\ddots  &\vdots&\vdots&&\vdots&\vdots&&\vdots&\vdots\\
		&&&p^2 & pb_{k+1}& \cdots & pb_{\ell}&pb_{\ell+1}&\cdots & pb_{n-1}&1\\
		&&&&p&\cdots&0&0&\cdots&0&1\\
		&&&&&\ddots&\vdots&\vdots&&\vdots&\vdots\\
		&&&&&&p^2&pc_{\ell+1}&\cdots&pc_{n-1}&1\\
		&&&&&&&p & \cdots &0&1\\
		&&&&&&&&\ddots &\vdots&\vdots\\
		&&&&&&&&&p&1
	\end{pmatrix}.
	$$
	\vspace{.2in}
	
	We use the row reduction method to determine all conditions that must be satisfied in order for this matrix to be an irreducible subring matrix. We obtain the conditions\\
	\begin{align*}
		&\frac{a_k(b_i^2 - b_i)}{p} \in \Z \text{ for } k < i < \ell\\
		&\frac{a_kb_\ell^2}{p} \in \Z \\
		& \frac{b_\ell(c_i^2 -c_i)}{p} \in \Z \text{ for } i > \ell\\
		& \frac{a_k b_\ell (c_i^2 - c_i)}{p^2} -\frac{a_\ell(c_i^2 - c_i)}{p} - \frac{a_k(b_i^2 - b_i)}{p} \in \Z \text{ for } i > \ell\\
		&\frac{a_k b_i b_j}{p} \in \Z \text{ for } k < i < j < \ell\\
		&\frac{a_k b_\ell(b_j - c_j)}{p} \in \Z \text{ for } j  > \ell\\	
		&\frac{b_\ell c_i c_j}{p} \in \Z \text{ for } \ell < i < j\\
		&\frac{a_kb_\ell c_i c_j}{p^2} - \frac{a_\ell c_i c_j}{p} - \frac{a_k b_i b_j}{p} \in \Z \text{ for } \ell < i < j.
	\end{align*}
	\vspace{.1in}
	
	Note that $\frac{a_kb_\ell^2}{p} \in \Z$ implies that either $p \mid a_k$ or $p \mid b_\ell$. We determine the number of solutions to three different sets of equations : (1) $p \mid a_k$, (2) $p \mid b_\ell$, and (3) $p$ divides both. Once we make these substitutions, the denominators are either 1 or $p$. We omit any term with denominator 1, since this term is already in $\Z$ and we see that all $a_i, b_i$ and $c_i$ are determined by their value $\pmod{p}$ since $a_i, b_i, c_i \in [0,p)$. We will reframe the problem in terms of counting $\F_p$-points on varieties since each denominator remaining is equal to $p$. 
	\vspace{.2in}
	
	\noindent 	\textbf{All equations when $p \mid a_k$ ($V_1$):} Since $a_k \in [0,p)$ then $p \mid a_k$ implies that $a_k =0$. We plug this into the conditions, which reduce to 
	\begin{align*}
		& \frac{b_\ell(c_i^2 -c_i)}{p} \in \Z \text{ for } i > \ell\\
		& \frac{a_\ell(c_i^2 - c_i)}{p}  \in \Z \text{ for } i > \ell\\
		&\frac{b_\ell c_i c_j}{p} \in \Z \text{ for } i > \ell, j > i\\
		& \frac{a_\ell c_i c_j}{p}  \in \Z\text{ for } i > \ell, j > i.
	\end{align*}
	Let $V_1$ be the variety over $\F_p$ defined by the polynomials 
		\begin{align*}
		&b_\ell(c_i^2 -c_i) \\
		& a_\ell(c_i^2 - c_i) \\
		&b_\ell c_i c_j \\
		& a_\ell c_i c_j
	\end{align*}
for all $\ell < i < j$.
	\vspace{.2in}
	
	\noindent\textbf{All equations when $p \mid b_\ell$ $(V_2)$:}
	As above, $p \mid b_\ell$ implies that $b_\ell =0$. We plug this into the conditions, which reduce to 
	\begin{align*}
		&\frac{a_k(b_i^2 - b_i)}{p} \in \Z \text{ for } k < i < \ell\\
		& \frac{a_\ell(c_i^2 - c_i)}{p}  - \frac{a_k(b_i^2 - b_i)}{p} \in \Z \text{ for } i > \ell\\
		&\frac{a_k b_i b_j}{p} \in \Z \text{ for } k < i < \ell, j > i\\
		& \frac{a_\ell c_i c_j}{p} - \frac{a_k b_i b_j}{p} \in \Z \text{ for } i > \ell, j > i.\end{align*}
	Let $V_2$ be the variety over $\F_p$ whose defining polynomials are 
		\begin{align*}
		&a_k(b_i^2 - b_i)\\
		&a_k b_i b_j 
		\end{align*}
	for $k < i < \ell $ and $i < j$, and 
		\begin{align*}
		& a_\ell(c_i^2 - c_i)-a_k(b_i^2 - b_i) \\
		& a_\ell c_i c_j - a_k b_i b_j\end{align*}
	for $\ell < i < j.$
	\vspace{.2in}
	
	\noindent \textbf{All equations when $p \mid a_k$ and $p \mid b_\ell$ ($V_3)$:}
	As before, we set $a_k = b_\ell =0$. The conditions reduce to 
	\begin{align*}
		& \frac{a_\ell(c_i^2 - c_i)}{p}  \in \Z \text{ for } i > \ell\\
		&\frac{a_\ell c_i c_j}{p}  \in \Z \text{ for } i > \ell, j > i.
	\end{align*}
	Let $V_3$ be the variety in $\F_p$ with defining polynomials $a_\ell(c_i^2 - c_i)$ and $a_\ell c_i c_j$ for $\ell < i < j$.
	
	The number of irreducible subring matrices with diagonal $\alpha$ is then
	$$ g_\alpha(p) = \#V_1(\F_p) + \#V_2(\F_p) - \#V_3(\F_p).
	$$
	
	\noindent \textbf{Claim 1:} We have
	$$
	\#V_1(\F_p) = p^{3n -k-\ell-9} + (n-\ell-1) (p-1) (p+1)p^{2n -k-7}.
	$$
	
	We summarize the cases for Claim 1 in Table \ref{t5}.
	\vspace{.1in}
	
	\begin{table}[!htbp]\centering
		\begin{tabular}{|c|c|c|c|}
			\hline
			Case &	$b_\ell$ & $a_\ell$ & Number of Solutions\\[6pt]
			\hline
			1&	0 & 0  & $p^{3n -k-\ell-9}$\\[6pt]
			\hline
			2&	0 & $\ne 0$ & $(n-\ell-1) (p-1) p^{2n -k-7}$\\[6pt]
			\hline
			3&	$\ne 0$ & 0 & $(n-\ell-1) (p-1) p^{2n -k-7}$\\[6pt]
			\hline
			4&	$\ne 0$ & $\ne 0$ & $(n-\ell-1) (p-1)^2 p^{2n - k-7}$\\[6pt]
			\hline
		\end{tabular}

		\caption{Cases when $p \mid a_k$}
			\label{t5}
	\end{table}
	
	\noindent \textbf{Claim 2:}	We have
	\begin{align*}
		\#V_2(\Fp) - \#V_3(\Fp)  &=(n-k-1)(p-1) p^{2n-\ell-6} + (\ell-k-1)(n-\ell-1)(p-1)^2p^{n-4} \\
		& - (n-\ell-1)(p-1)p^{n-4} \\
		& + \left((n-\ell-2) + \binom{n-\ell-2}{2}\right)(p-1)(p-2)(p-3)p^{n-4}.
	\end{align*}
	\vspace{.1in}
	
	We summarize the cases for Claim 2 in the Table \ref{t6}.
	
	\begin{table}[!htbp]\centering
		\begin{tabular}{|c|c|c|p{2cm}|p{1.7cm}|c|}
			\hline
			Case &$a_k$ & $a_\ell$ &$b_i$ & Other Conditions &  Number of Solutions\\[6pt]
			\hline
			1 & 0 &  & && \text{Not counted}\\
			\hline
			2&$\ne 0$&	0& && $(n-k-2) (p-1)p^{2n-\ell-6}$\\[6pt]
			\hline
			3A&	$\ne 0$ &$\ne 0$ & &$b_i =1$ for one $k < i < \ell$ & $(\ell-k-1)  (n-\ell-1)  (p-1)^2  p^{n-4}$\\[6pt]
			\hline
			3B(i)&$\ne 0$ &$\ne 0$ &$b_i =0$ for all $k < i< \ell$,$b_i \in \{0,1\}$ for all $i$ &Not all $b_i = c_i$ for $i  > \ell$&  $(n-\ell-1)^2 (p-1)^2p^{n-4}$\\[6pt]
			\hline
			3B(ii)a&$\ne 0$ &$\ne 0$ &$b_i =0$ for all $k < i < \ell$, at least one $b_j \ne 0$ for $j > \ell$ &$c_i=b_i$&$p^{n-\ell-2} p^{n-4}(p-1)  - (n-\ell-1) p^{n-4}(p-1)$  \\[6pt]
			\hline
			3B(ii)b&$\ne 0$ &$\ne 0$ &$b_i =0$ for all $k < i < \ell$, at least one $b_j \ne 0$ for $j > \ell$ &$c_j=0$ for all $j >\ell$& $(n-\ell-2)(p-1)(p-2)(p-3)p^{n-4}$ \\[6pt]
			\hline
			3B(ii)c&$\ne 0$ &$\ne 0$ &$b_i =0$ for all $k < i < \ell$, at least one $b_j \ne 0$ for $j > \ell$ &$c_j=1-c_i$ for exactly one $j > \ell$&$\binom{n-\ell-2}{2} (p-1)(p-2) (p-3) p^{n-4}$  \\[6pt]
			\hline
		\end{tabular}

		\caption{Cases when $p \mid b_\ell$ }
			\label{t6}
	\end{table}
	
	\noindent Together, Claims 1 and 2 prove the lemma.
	\vspace{.1in}
	
	\noindent \textbf{Proof of Claim 1:}
	
	First observe that there are no conditions on $a_i$ for $i \ne \ell, k$ and $b_i $ for $k < i < n-1$ and $i \ne \ell$. Therefore we multiply each case by $p^{n-4} \cdot p^{n-k-3} = p^{n-k-7}.$
	\vspace{.1in}
	
	\noindent \textbf{Proof of Case 1:} If $b_\ell = a_\ell = 0$ there are no conditions for any other variable. So we obtain
	$$
	p^{n-4} \cdot p^{n-k-3} \cdot p^{n-\ell-2} = p^{3n -k-\ell-9}
	$$
	solutions.
	\vspace{.1in}
	
	\noindent \textbf{Proof of Case 2:} If $b_\ell = 0$ but $a_\ell \ne 0$ we must solve the equations over $\F_p$:
	\begin{align*}
		& (c_i^2 - c_i) =0 \text{ for } i > \ell\\
		& c_i c_j =0 \text{ for } i > \ell, j > i.
	\end{align*}
	
	Either every $c_i = 0$ or there is exactly one $c_i =1$. There are $n -\ell -2$ ways to pick the $c_i = 1$ and there is one way to pick all $c_i =0$. Therefore we have
	$$
	(n-\ell-1) (p-1) p^{2n -k-7}
	$$
	solutions.
	\vspace{.1in}
	
	\noindent \textbf{Proof of Case 3:} If $a_\ell =0 $ but $b_\ell \ne 0$, we solve the same equations as in Case 2. So we have
	$$
	(n-\ell-1) (p-1) p^{2n - k-7}
	$$
	solutions.
	\vspace{.1in}
	
	\noindent \textbf{Proof of Case 4:} If $a_\ell \ne 0$ and $b_\ell \ne 0$, we solve the same equations as in Case 2 and now there are $p-1$ choices for $a_\ell$ and $p-1$ choices for $b_\ell$ giving
	$$
	(n-\ell-1) (p-1)^2 p^{2n - k-7}
	$$
	solutions.
	\vspace{.1in}
	
	In total we have
	$$
	\#V_1(\Fp) = p^{3n -k-\ell-9} + 2(n-\ell-1) (p-1) p^{2n - k-7} + (n-\ell-1) (p-1)^2 p^{2n - k-7}.
	$$
	We can simplify to
	$$
	\#V_1(\Fp) = p^{3n -k-\ell-9} + (n-\ell-1) (p-1) (p+1)p^{2n -k-7}.
	$$
	
	\vspace{.1in}
	
	\noindent \textbf{Proof of Claim 2:}
	Observe that $a_i$ for $i \ne k, \ell$ is arbitrary, so we multiply each case by $p^{n-4}$. Recall that $b_\ell =0$.
	\vspace{.1in}
	
	\noindent \textbf{Proof of Case 1:} Suppose $a_k=0$. We have already found the solutions when $a_ k= b_\ell =0$, so we don't count it again here. This allows us to count $\#V_2(\F_p)  - \#V_3(\F_p)$ in Claim 2.
	\vspace{.1in}
	
	\noindent \textbf{Proof of Case 2:} Suppose $a_\ell =0, a_k \ne 0$. The equations in $\F_p$ reduce to
	\begin{align*}
		&(b_i^2 - b_i) =0 \text{ for } k < i \\
		&b_i b_j =0 \text{ for } k < i <j.
	\end{align*}
	Either all $b_i =0$ or there is exactly one $b_i =1$. In either case, we see that $c_i$ for $\ell < i < n-1$ are arbitrary. There are $n-k-3$ ways to pick the $b_i =1$ (we know $b_\ell=0$) and there is one way to pick all $b_i =0$. We thus have
	$$
	(n-k-2)(p-1)p^{n-4} \cdot 1 \cdot p^{n-\ell -2} =(n-k-2) (p-1)p^{2n-\ell-6}
	$$
	solutions.
	\vspace{.1in}
	
	\noindent \textbf{Proof of Case 3:}  Suppose both $a_\ell \ne 0$ and $a_k \ne 0$. Start with the equations in $\F_p$
	\begin{align*}
		&(b_i^2 - b_i) =0\text{ for } k < i < \ell\\
		&b_i b_j =0\text{ for } k < i < \ell, j > i.
	\end{align*}
	
	\noindent \textbf{Proof of Case 3A:} Suppose exactly one $b_i=1$ for some $k < i< \ell$, so all $b_i = 0$ for $i \ne k, \ell$. There are $(\ell-k-1)$ ways to choose an index $k < i< \ell$ so that $b_i =1$. The equations defining $V_2$ over $\F_p$ then reduce to 
	\begin{align*}
		& a_\ell(c_i^2 - c_i)=0 \text{ for } i > \ell\\
		& a_\ell c_i c_j =0\text{ for } i > \ell, j > i.
	\end{align*}
	There are $(n-\ell-1)$ ways to choose at most one $c_i =1$. Therefore we have a total of 
	$$
	(\ell-k-1)  (n-\ell-1)  (p-1)^2  p^{n-4}
	$$
	solutions.
	\vspace{.1in}
	
	\noindent \textbf{Proof of Case 3B:} Suppose $b_i =0$ for $k < i < \ell$. The equations defining $V_2$ over $\F_p$ reduce to
	\begin{align}
		& a_\ell(c_i^2 - c_i) -a_k(b_i^2 - b_i) =0 \text{ for } i > \ell \label{eqnC3.1}\\
		& a_\ell c_i c_j -a_k b_i b_j=0 \text{ for } i > \ell, j > i \label{eqnC3.2}.
	\end{align}
	
	Observe that $b_i = 0,1$ if and only if $c_i =0,1$.
	\vspace{.1in}
	
	\noindent \textbf{Proof of Case 3B(i):} First suppose that all $b_i \in \{0,1\}$ and $b_i =0$ for $k < i< \ell$. If $b_i = b_j = 1$, then by \eqref{eqnC3.2}, $a_\ell c_i c_j = a_k$ and so $c_i c_j = a_k a_\ell^{-1}$. Since $c_i, c_j \in \{0,1\}$ we must have that $c_i  = c_j = 1$ and so $a_k = a_\ell$. Therefore $c_m = b_m$ for all $m > \ell$ and $a_k = a_\ell$. We count this case later in Case 3B(ii)a.
	
	If at most one $b_i =1$ then at most $c_j =1$. Therefore we have
	$$
	(n-\ell-1)^2 (p-1)^2p^{n-4}
	$$
	solutions.
	\vspace{.1in}
	
	\noindent \textbf{Proof of Case 3B(ii):}
	Suppose that at least one $b_i \ne 0,1$. Then we can solve for $a_k = \frac{a_\ell (c_i^2 - c_i)}{(b_i^2-b_i)}$ using \eqref{eqnC3.1}. We plug this into \eqref{eqnC3.2} to find
	$$
	a_\ell c_i c_j - \frac{a_\ell (c_i^2 - c_i)}{(b_i^2-b_i)}b_ib_j =0.
	$$
	Since $c_i \ne 0,1$, simplifying gives
	\begin{equation} \label{eqn2bj}
		b_j = \frac{c_j(b_i-1)}{(c_i-1)}.
	\end{equation}
	
	Now plug these equations into \eqref{eqnC3.2} to obtain the expression
	$$
	a_\ell(c_j^2 -c_j) - \frac{a_\ell (c_i^2 - c_i)}{(b_i^2-b_i)}\left(\frac{c_j^2(b_i-1)^2}{(c_i-1)^2} - \frac{c_j(b_i-1)}{(c_i-1)}\right) =0.
	$$
	
	This reduces to
	$$
	a_\ell(c_j^2 -c_j)  - \frac{a_\ell c_i c_j^2 (b_i-1)}{b_i (c_i-1)} + \frac{a_\ell c_i c_j}{b_i } =0.
	$$
	
	Clearing denominators,
	$$
	a_\ell c_j (c_j-1) (c_i-1) b_i - a_\ell c_i c_j^2 (b_i-1) + a_\ell c_i c_j (c_i-1) =0
	$$
	and so 
	$$
	a_\ell c_j \left((c_j-1)(c_i-1)b_i - c_i c_j (b_i-1) + c_i(c_i-1) \right) =0
	$$
	which reduces to
	$$
	a_\ell c_j \left(c_jc_ib_i - c_i b_i - c_jb_i + b_i - c_i c_j b_i+ c_ic_j +c_i^2 - c_i \right) =0.
	$$
	Finally, we simplify this equation to
	\begin{equation} \label{eqncj}
		a_\ell c_j (c_i -b_i)(c_j + c_i - 1)  =0.
	\end{equation}
	
	We know that $a_\ell \ne 0$. Further, since at least one of the terms in the product must be equal to 0, then the $c_j$ are determined. Thus it remains to solve $a_\ell (c_i^2-c_i) - a_k (b_i^2-b_i) =0. $
	
	Suppose that $c_j = 1-c_i$ and $c_m = 1-c_i$ for $ j \ne m > \ell$. Then plugging into \eqref{eqn2bj}, we see $b_j  = b_m = 1-b_i$. Recall that $a_\ell = a_k \frac{b_i^2 - b_i}{c_i^2-c_i}$ from \eqref{eqnC3.1} and $a_\ell=a_k \frac{b_jb_m}{c_jb_m}$ from \eqref{eqnC3.2}. Therefore setting these equal to each other and substituting gives
	$$
	b_i(c_i-1) = c_i (b_i-1)
	$$
	which implies that $c_i = b_i$. Thus $a_\ell = a_k$ and $b_j  = c_j$ for all $j > i$. We solve this case below in Case 3B(ii)a.
	\vspace{.in}
	
	\noindent \textbf{Proof of Case 3B(ii)a:}
	Suppose $c_i = b_i$. Then $b_j = c_j$ for all $j > i$ and $a_k = a_\ell$. The $c_j \in [0,p)$, $a_k = a_\ell$, and $b_j$ are all determined by $c_j$. The variable $a_\ell \in (0,p)$. This gives
	$$
	p^{n-\ell-2} p^{n-4}(p-1)  - (n-\ell-1) p^{n-4}(p-1)
	$$
	solutions discounting the case when all $b_j, c_j$ are in $\{0,1\}$.
	\vspace{.1in}
	
	\noindent \textbf{Proof of Case 3B(ii)b:}
	Suppose all $c_j =0$. We are left to solve $a_\ell (c_i^2 - c_i) -a_k (b_i^2 -b_i) =0.$  Observe that there are $(n -\ell -2)$ ways to pick the $i$ so that $b_i, c_i \ne 0$. We can solve for $a_\ell$ and then pick $a_k \ne 0$. There are $(p-2)$ ways to pick $c_i \ne 0,1$ and $(p-3)$ ways to pick $b_i \ne c_i, 0,1$. This gives
	$$
	(n-\ell-2)(p-1)(p-2)(p-3)p^{n-4}
	$$
	solutions.
	\vspace{.1in}
	
	\noindent \textbf{Proof of Case 3B(ii)c:}
	Suppose exactly one $c_j = 1-c_i$. Then the rest of $c_m = 0$ for $m \ne i, j$. We are left to solve $a_\ell (c_i^2 - c_i) -a_k (b_i^2 -b_i) =0.$ Again, $a_\ell$ is determined and $a_k \ne 0$. We can choose $i,j$ so that $c_i \ne 0,1$ and $c_j = 1-c_i$. There are $p-2$ choices for $c_i$ and $p-3$ choices for $b_i$. The rest of the variables are fixed. We have
	$$
	\binom{n-\ell-2}{2} (p-1)(p-2) (p-3) p^{n-4}
	$$
	solutions.
	\vspace{.2in}
	
	In total, 
	\begin{align*}
		\#V_2(\Fp) - \#V_3(\Fp)  &=(n-k-1)(p-1) p^{2n-\ell-6} + (n-k-2)(n-\ell-1)(p-1)^2p^{n-4}\\
		& - (n-\ell-1)(p-1) p^{n-4} \\
		&+ \left((n-\ell-2) + \binom{n-\ell-2}{2}\right)(p-1)(p-2)(p-3)p^{n-4}.
	\end{align*}
	
\end{proof}

\begin{proof}[Proof of Corollary \ref{cor222}]
	We prove this summation one term at a time using the terms in Lemma \ref{lemma222}. Each summation is found using Mathematica \cite{mathematica}.
\end{proof}

\bibliographystyle{habbrv}
\bibliography{../../Bibliography/bib_all}

\end{document}